\documentclass[11 pt]{amsart}
\usepackage{latexsym,amscd,amssymb, graphicx, amsthm}
\usepackage{tikz}

\usepackage[margin=1in]{geometry}

\numberwithin{equation}{section}

\newtheorem{theorem}{Theorem}[section]

\newtheorem{corollary}[theorem]{Corollary}
\newtheorem{lemma}[theorem]{Lemma}

\newtheorem{problem}[theorem]{Problem}

\newtheorem{remark}[theorem]{Remark}

\newtheorem{definition}[theorem]{Definition}
\theoremstyle{definition}

\newcommand{\comaj}{\mathrm {comaj}}
\newcommand{\maj}{\mathrm {maj}}

\newcommand{\Des}{\mathrm {Des}}

\newcommand{\des}{\mathrm {des}}
\newcommand{\asc}{\mathrm {asc}}

\newcommand{\Hilb}{\mathrm {Hilb}}
\newcommand{\symm}{\mathfrak{S}}
\newcommand{\grFrob}{\mathrm {grFrob}}
\newcommand{\Frob}{\mathrm {Frob}}

\newcommand{\CC}{\mathbb {C}}

\newcommand{\ZZ}{\mathbb {Z}}

\newcommand{\OP}{\mathcal{OP}}

\newcommand{\AAA}{\mathcal{A}}
\newcommand{\DDD}{\mathcal{D}}
\newcommand{\EEE}{\mathcal{E}}
\newcommand{\FFF}{\mathcal{F}}
\newcommand{\III}{\mathcal{I}}
\newcommand{\JJJ}{\mathcal{J}}
\newcommand{\RRR}{\mathcal{R}}
\newcommand{\SSS}{\mathcal{S}}
\newcommand{\UUU}{\mathcal{U}}
\newcommand{\VVV}{\mathcal{V}}

\newcommand{\GL}{\mathrm {GL}}

\newcommand{\xx}{\mathbf {x}}
\newcommand{\yy}{\mathbf {y}}

\newcommand{\cB}{\mathcal{B}}

\begin{document}

\title[Generalized coinvariant algebras for $G(r,1,n)$ in the Stanley--Reisner setting]
{Generalized coinvariant algebras for $G(r,1,n)$ in the Stanley--Reisner setting}

\author{Dani\"el Kroes}
\address
{Department of Mathematics \newline \indent
University of California, San Diego \newline \indent
La Jolla, CA, 92093-0112, USA}
\email{dkroes@ucsd.edu}

\begin{abstract}
Let $r$ and $n$ be positive integers, let $G_n$ be the complex reflection group of $n \times n$ monomial matrices whose entries are $r^{\textrm{th}}$ roots of unity and let $0 \leq k \leq n$ be an integer. Recently, Haglund, Rhoades and Shimozono ($r=1$) and Chan and Rhoades ($r>1$) introduced quotients $R_{n,k}$ (for $r>1$) and $S_{n,k}$ (for $r \geq 1$) of the polynomial ring $\mathbb{C}[x_1,\ldots,x_n]$ in $n$ variables, which for $k=n$ reduce to the classical coinvariant algebra attached to $G_n$. When $n=k$ and $r=1$, Garsia and Stanton exhibited a quotient of $\mathbb{C}[\mathbf{y}_S]$ isomorphic to the coinvariant algebra, where $\mathbb{C}[\mathbf{y}_S]$ is the polynomial ring in $2^n-1$ variables whose variables are indexed by nonempty subsets $S \subseteq [n]$. In this paper, we will define analogous quotients that are isomorphic to $R_{n,k}$ and $S_{n,k}$.
\end{abstract}

\maketitle

\section{Introduction}
\label{Introduction}

In this paper we will study recently introduced generalizations of the \emph{coinvariant algebra} attached to the complex reflection group $G(r,1,n)$ of $r$-colored permutations of $[n] = \{1,2,\ldots,n\}$. In the case $r = 1$, when this group is the symmetric group $\symm_n$, the coinvariant algebra is defined as follows. Let $\xx_n = (x_1,\ldots,x_n)$ be a list of $n$ variables and let $\symm_n$ act on $\CC[\xx_n]$ by permutation of the variables. The \emph{ring of symmetric functions} is the corresponding invariant subring
\[
\CC[\xx_n]^{\symm_n} = \{f \in \CC[\xx_n] \ : \ \sigma \cdot f = f \textup{ for all } \sigma \in \symm_n\},
\]
and the \emph{invariant ideal} is the ideal $I_n = \langle \CC[\xx_n]^{\symm_n}_+ \rangle \subseteq \CC[\xx_n]$ generated by all the symmetric functions with vanishing constant term. The \emph{coinvariant algebra} attached to $\symm_n$ is the quotient $R_n := \CC[\xx_n]/I_n$.

It is well known that the ring of symmetric functions has algebraically independent homogeneous generators $e_1(\xx_n), \ldots, e_n(\xx_n)$ where $e_d(\xx_n)$ is the \emph{elementary symmetric function of degree $d$} defined by
\[
e_d(\xx_n) = \sum_{1 \leq i_1 < \ldots < i_d \leq n} x_{i_1} \cdots x_{i_d}.
\]
Using this we may write $I_n$ more succinctly as $I_n = \langle e_1(\xx_n),\ldots,e_n(\xx_n) \rangle$ and consequently we have
\[
R_n = \frac{\CC[\xx_n]}{\langle e_1(\xx_n),\ldots,e_n(\xx_n) \rangle}.
\]

Note that $R_n$ is a graded $\symm_n$-module. By Chevalley's theorem \cite{C}, it is known that as an ungraded module $R_n$ is isomorphic to $\CC[\symm_n]$, the regular representation of $\symm_n$. Furthermore, $R_n$ has Hilbert series given by
\[
\Hilb(R_n,q) = [n]!_q = \sum_{\sigma \in \symm_n} q^{\maj(\sigma)},
\]
where $\maj$ is the \emph{major index statistic} on $\symm_n$. This raises the question of whether there is a vector space basis for $R_n$ whose elements $\xx_{\sigma}$ are naturally indexed by permutations $\sigma \in \symm_n$ in such a way $\xx_{\sigma}$ has degree $\maj(\sigma)$, a question that was affirmatively answered by Garsia \cite{G}. Furthermore, Garsia and Stanton \cite{GS} exhibited the coinvariant algebra as a quotient of the polynomial ring $\CC[\yy_S]$ where $\yy_S = \{y_{\{1\}}, \ldots, ,y_{\{n\}}, \ldots, y_{\{1,2,\ldots,n\}}\}$ is a list of variables indexed by the non-empty subsets $S \subseteq [n]$. In particular, they showed an isomorphism between $R_n$ and
\[
\RRR_n = \frac{\CC[\yy_S]}{\langle y_S \cdot y_T, \theta_1, \ldots, \theta_n \rangle},
\]
where the generators of the ideal in the denominator come in two types; firstly we have $y_S \cdot y_T$ for non-empty subsets $S,T \subseteq [n]$ with $S \not \subseteq T$ and $T \not \subseteq S$ and secondly we have $\theta_i = \sum_{S \subseteq [n], |S| = i} y_S$ for $1 \leq i \leq n$. Furthermore, the quotient $\RRR_n$ has a similar Garsia--Stanton basis that witnesses not only the major index statistic on $\symm_n$, but also the descent statistic.

Recently, Haglund, Rhoades and Shimozono \cite{HRS} defined a generalization $R_{n,k}$ of $R_n$ for any pair of integers $0 \leq k \leq n$ with $n \geq 1$. More specifically, they set
\[
I_{n,k} = \langle e_{n-k+1}(\xx_n), \ldots, e_n(\xx_n), x_1^k, \ldots, x_n^k \rangle \qquad \textup{and} \qquad R_{n,k} = \frac{\CC[\xx_n]}{I_{n,k}}.
\]
When $n = k$ the quotient $R_{n,n}$ coincides with the classical coinvariant algebra $R_n$. In their paper, they show that as an ungraded $\symm_n$-module we have $R_{n,k} \cong \CC[\OP_{n,k}]$, where $\OP_{n,k}$ is the set of ordered set partitions of $[n]$ with $k$ blocks. Furthermore, there exists a statistic $\comaj$ on $\OP_{n,k}$, which reduces to the complement of $\maj$ when $k = n$, and the Hilbert series of $R_{n,k}$ is given by
\[
\Hilb(R_{n,k},q) = \sum_{\pi \in \OP_{n,k}} q^{\comaj(\pi)}.
\]
In addition to this, they exhibit a Garsia--Stanton type basis for $R_{n,k}$ witnessing this fact.

In this paper we will prove that these generalized coinvariant algebras also have similar quotients of $\CC[\yy_S]$, in the sense that we have an $S_n$-isomorphism between $R_{n,k}$ and $\RRR_{n,k}$, where $\RRR_{n,k}$ is defined as below. We remark that when $n = k$ the quotient $\RRR_{n,n}$ coincides with the quotient $\RRR_n$ introduced by Garsia and Stanton \cite{GS}.
\begin{definition}
\label{intro-definition}
Let $0 \leq k \leq n$ be integers with $n \geq 1$. In $\CC[\yy_S]$ we define the following ideal:
\begin{equation*}
\III_{n,k} = \langle y_S \cdot y_T, \theta_{n-k+1},\ldots,\theta_n, y_{S_1} \cdots y_{S_k} \rangle,
\end{equation*}
where $S$ and $T$ range over all pairs of nonempty subsets $S,T \subseteq [n]$ with $S \not \subseteq T$ and $T \not \subseteq S$,
\[
\theta_i = \sum_{S \subseteq [n], |S| = i} y_S
\]
and $(S_1,\ldots,S_k)$ ranges over all length $k$ multichains $S_1 \subseteq \ldots \subseteq S_k$ of nonempty subsets of $[n]$.

Finally, set $\RRR_{n,k} = \CC[\yy_S]/\III_{n,k}$.
\end{definition}
Additionally, we will show that the Garsia--Stanton type basis in Haglund, Rhoades and Shimozono \cite{HRS} has an equivalent natural basis of $\RRR_{n,k}$. Our methods are inspired by Braun and Olsen \cite{BO}, who obtain a similar result when $n = k$. In theory, our method gives a way to define a descent statistic on $\OP_{n,k}$, although we currently do not have a nice combinatorial interpretation for this.

Furthermore, when $r > 1$ Chan and Rhoades \cite{CR} define two generalizations of the coinvariant algebra attached to $G(r,1,n)$. Similarly to the case $r=1$, one of these modules is isomorphic to $\CC[\OP^{(r)}_{n,k}]$, where $\OP_{n,k}^{(r)}$ is the set of block ordered set partitions of $[n]$ with $k$ blocks, where each element of $[n]$ is assigned one of $r$ colors. The other module is isomorphic to $\CC[\FFF_{n,k}^{(r)}]$, where $\FFF_{n,k}^{(r)}$ is the set of $k$-dimensional faces attached to the Coxeter complex of $G(r,1,n)$ which can be thought of as $r$-colored ordered set partitions of some set $S \subseteq [n]$ into $k$ blocks. Again, the grading of these modules is controlled by the $\comaj$-statistic on $\FFF_{n,k}$ and $\OP_{n,k}$, and they describe a Garsia-Stanton type basis in this case as well. Just as for $r = 1$ we will show an isomorphism between their modules and an appropriate quotient of $\CC[\yy_S]$ and show that this Garsia--Stanton type basis has an analogous natural basis of our quotients.

In section \ref{Background} we will review the necessary background concerning the classical coinvariant algebra of $G(r,1,n)$, ordered set partitions and their $\maj$ and $\comaj$ statistics and a bit of Gr\"obner theory. In section \ref{Bases} we will define the appropriate analogues of Definition \ref{intro-definition} for all $r \geq 1$ and deduce the analogues of the Garsia--Stanton bases. In section \ref{Filtration} we will use filtrations on $R_{n,k}$, $S_{n,k}$, $\RRR_{n,k}$ and $\SSS_{n,k}$ to deduce the desired isomorphisms. In section \ref{Frobenius} we will discuss the multi-graded Frobenius image of the modules in the case of the symmetric group $\symm_n$. In section \ref{conclusion} we will conclude with some remarks and possible future directions.

\section{Background}
\label{Background}

\subsection{Ordered set partitions and $G_n$ faces}
Fix an integer $r \geq 1$ that will be suppressed for the remainder of this paper. Let us introduce $G(r,1,n)$ which from now on will be denoted by $G_n$. Recall that $G_n \subseteq \GL_n(\CC)$ is the reflection group consisting of all monomial matrices with entries in $\{1,\zeta,\zeta^2,\ldots,\zeta^{r-1}\}$ where $\zeta = \exp(2\pi i/r)$ is a primitive $r$-th root of unity. Matrices in $G_n$ can be thought of as \emph{$r$-colored permutations} $\pi_1^{c_1} \cdots \pi_n^{c_n}$ where $\pi_1 \cdots \pi_n \in \symm_n$ and $c_1, \ldots, c_n \in \{0,1,\ldots,r-1\}$ indicate the colors of the $\pi_i$. For example, when $r = 4$, the element
\[
\begin{pmatrix}
0 & i & 0 & 0 & 0 \\
0 & 0 & 0 & -1 & 0 \\
-i & 0 & 0 & 0 & 0 \\
0 & 0 & 0 & 0 & 1 \\
0 & 0 & -1 & 0 & 0
\end{pmatrix}
\in G_5
\]
will be interpreted as $3^3 1^1 5^2 2^2 4^0$.

We will consider the alphabet $\AAA_r = \{i^c \ : \ i \in [n], c \in \{0,1,\ldots,r-1\}\}$ consisting of letters $i$ with an associated color $c$. On this alphabet we put an order $<$ that weighs colors more heavily than letters in the following way:
\[
1^{r-1} < 2^{r-1} < \ldots < n^{r-1} < 1^{r-2} < 2^{r-2} < \ldots < 1^0 < 2^0 < \ldots < n^0.
\]
Given a word $w = w_1^{c_1} \cdots w_r^{c_r}$ in $\AAA_r$ we define the \emph{descent set} of $w$ by
\begin{equation}
\Des(w) = \{1 \leq i \leq n-1 \ : \ w_i^{c_i} > w_{i+1}^{c_{i+1}}\},
\end{equation}
and we will write $\des(w) := |\Des(w)|$. The \emph{major index} of $w$ is given by
\begin{equation}
\maj(w) = c(w) + r \cdot \sum_{i \in \Des(w)} i,
\end{equation}
where $c(w)$ denotes the sum of the colors of $w$. Given our interpretation of elements of $G_n$ as $r$-colored permutations these definitions directly apply to elements of $g \in G_n$. For example, when $w = 3^3 1^1 5^2 2^2 4^0$ as above we have $\Des(w) = \{2,3\}$ and $\maj(w) = (3 + 1 + 2 + 2 + 0) + 4 \cdot (2 + 3) = 28$.

An \emph{ordered set partition} of $[n]$ is a partition of $[n]$ with a total order on the blocks. An \emph{$r$-colored ordered set partition} is
an ordered set partition of $[n]$ with a color associated to each of the elements. For $1 \leq k \leq n$, let $\OP_{n,k}$ be the set of $r$-colored ordered set partitions of $[n]$ with $k$ blocks. Elements within a block will usually be written in increasing order with respect to $<$. When $r = 4$, an example of an element in $\OP_{9,5}$ is given by
\[
\{4^3,2^2,3^2\} \prec \{9^1\} \prec \{6^1,1^0\} \prec \{5^2\} \prec \{7^2,8^1\}.
\]
Throughout this paper, we will use the \emph{ascent-starred model} for ordered set partitions. This means that we will write down all the elements in the ordered set partition in the order they appear and use stars to indicate which elements should be grouped together in a block. In this model, the above example will be represented as
\[
4^3_* 2^2_* 3^2 \ 9^1 \ 6^1_*1^0 \ 5^2 \ 7^2_*8^1.
\]
Alternatively, we can represent this as a pair $(g,\lambda)$ where $g \in G_n$ satisfies $\des(g) < k$ and $\lambda$ is a partition with $k - \des(g) - 1$ parts, all at most $n-k$, that indicates which of the ascents should be starred. In the above example, we would have $g = 4^3 2^2 3^2 9^1 6^1 1^0 5^2 7^2 8^1$, which has $\Des(g) = \{4,6\}$, so $\des(g) = 2$ and of the $6$ ascents, the first, second, fourth and sixth are starred. Representing this as a path from $(0,2)$ to $(2,4)$, where each starred ascent contributes a step east and each unstarred ascent contributes a step south, we get the following picture
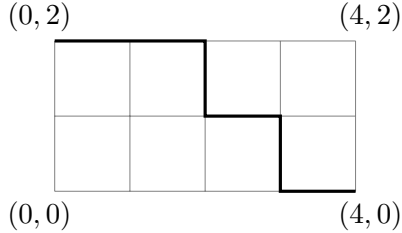
\[
\begin{tikzpicture}
\draw[step=1cm,gray,very thin] (0,0) grid (4,2);
\draw[very thick] (0,2) -- (2,2) -- (2,1) -- (3,1) -- (3,0) -- (4,0);
\node[below] at (-0.2,0) {$(0,0)$};
\node[above] at (-0.2,2) {$(0,2)$};
\node[below] at (4.2,0) {$(4,0)$};
\node[above] at (4.2,2) {$(4,2)$};
\end{tikzpicture}
\]
so the corresponding partition $\lambda$ is $(3,2)$ the representative is $(4^3 2^2 3^2 9^1 6^1 1^0 5^2 7^2 8^1,(3,2))$.

\begin{definition}
Let $(g,\lambda) \in \OP_{n,k}$. Define
\[
\comaj((g,\lambda)) = \maj(g) + r |\lambda|.
\]
\end{definition}

\begin{remark}
\begin{itemize}
\item[1.] In the case $n = k$ this definition actually swaps the notion of $\maj$ and $\comaj$ on $G_n$. This means that we have to look at ascents rather than descents, but this does not change the combinatorial flavor of the statistic, nor its distribution.
\item[2.] In Lemma \ref{comajr1} below we will show that when $r = 1$ this notion of $\comaj$ coincides with the $\comaj$-statistic defined in Haglund, Rhoades and Shimozono \cite{HRS}. However, for $r > 1$ this notion of $\comaj$ is not complementary to the $\maj$-statistic defined in Chan and Rhoades \cite{CR}, as they use a descent-starred model for ordered set partitions. Furthermore, their $\comaj$-statistic includes a sum over the complements of the colors of the ordered set partition, whereas our definition includes simply a sum over the colors of the ordered set partition. However, these definitions do not change the idea behind the ordered set partition and each such choice to be made does not affect the eventual distribution of the statistic.
\end{itemize}
\end{remark}

\begin{lemma}
\label{comajr1}
Let $r = 1$ and let $\maj((g,\lambda))$ be as in equation (2.4) in \cite{HRS}. Then we have
\[
\comaj((g,\lambda)) = (n-k)(k-1) + \binom{k}{2} - \maj((g,\lambda))
\]
\end{lemma}

\begin{proof}
Note that $\maj((g,\lambda))$ equals the sum over all the \emph{ascents} of $\sigma$ with respect to the weight sequence $(w_1,\ldots,w_n)$, where $w_i$ is the number of completed blocks when reaching the $i^{\textrm{th}}$ element of $g$. For example, when $(g,\lambda) = (2461357,(1,1))$, corresponding to $(24|6|1|357)$, the weight sequence is given by $(0,1,2,3,3,3,4)$. Now, $\comaj(g)$ is the sum over all the \emph{ascents} of $\sigma$ with respect to the weight sequence $(1,2,\ldots,n)$. Now, every starred ascent of $(g,\lambda)$ results in the weights of all the ascents after (and including) that ascent to be decreased by $1$. Now remember that $\lambda$ is a permutation of which each part has length at most $n-k$, and the stars correspond to the horizontal segments in the left bottom justified Ferrers diagram. Therefore, the number of affected ascents equals
\begin{align*}
(1 + &\textup{the height of the last column})+ (2 + \textup{the height of the second to last column}) + \ldots \\ & \quad + ((n-k) + \textup{the height of the first column}) = (1 + 2 + \ldots + (n-k)) + |\lambda|.
\end{align*}
Therefore, we have
\[
\maj((g,\lambda)) = \comaj(g) - (1 + \ldots + (n-k)) - |\lambda|,
\]
so we can compute
\begin{align*}
\comaj((g,&\lambda)) = (1 + \ldots + (k-1)) + (n-k)(k-1) - \maj((g,\lambda)) \\
    &= (1 + \ldots + (k-1)) + (n-k)(k-1) + (1 + \ldots + (n-k)) - \comaj(g) + |\lambda| \\
    &= (1 + \ldots + (k-1)) + (k + \ldots + (n-1)) - \comaj(g) + |\lambda| \\
    & = 1 + 2 + \ldots + (n-1) - \comaj(g) + |\lambda| = \maj(g) + |\lambda|. \qedhere
\end{align*}
\end{proof}

We will now introduce a similar model for $G_n$ faces. We first recall a definition of Chan and Rhoades \cite[Definition. 2.1]{CR}.
\begin{definition}
Let $r \geq 2$. A \emph{$G_n$-face} is an ordered set partition $\sigma = (B_1 \ | \ B_2 \ | \ \cdots \ | \ B_m)$ of $[n]$ where each number is assigned a color from $\{0,1,\ldots,r-1\}$, except for possibly (all) the elements in $B_1$.

If the numbers in $B_1$ are colored, we say that $\sigma$ has dimension $m$ and if the numbers are uncolored we say that $\sigma$ has dimension $m-1$. We define $\FFF_{n,k}$ to be the set of $G_n$-faces of dimension $k$.
\end{definition}

If the elements of $B_1$ are uncolored, $B_1$ is referred to as the \emph{zero-block} of $\sigma$. For example, when $r = 3$ we have that $(14 \ | \ 5^2 2^1 3^0 \ | \ 7^2 \ | \ 6^0) \in \FFF_{7,3}$ and $(1^2 4^2 \ | \ 5^2 2^1 3^0 \ | \ 7^2 \ | \ 6^0) \in \FFF_{7,4}$. We can represent elements of $\FFF_{n,k}$ in a way similar to our ascent-starred model for ordered set partitions. Elements $\sigma \in \FFF_{n,k}$ will be represented as triples $(Z,g,\lambda)$ where $Z \subseteq [n]$ has size $|Z| \leq n-k$, $g$ is a word in which each of the letters of $[n] \backslash Z$ appears once and is labeled with a color in $\{0,1,\ldots,r-1\}$, $\des(g) < k$ and $\lambda$ is a partition with $k - \des(g) - 1$ parts that are all at most $n - |Z| - k$. Here, $Z$ represents of the elements of $[n]$ that belong to the (possibly empty) zero-block of $\sigma$, and $(g,\lambda)$ will yield the remaining block ordered set partition using a similar process to before, by choosing which ascents of $g$ to star according to $\lambda$. Note that $g$ might not include every number in $[n]$, but that the same process still works. For example, the two examples above will be represented by $(\{1,4\},5^22^13^17^26^0,(2))$ and $(\emptyset,1^24^15^22^13^17^26^0,(3,1))$ respectively.

The $\comaj$-statistic on $\FFF_{n,k}$ is defined as follows.
\begin{definition}
Let $(Z,g,\lambda) \in \FFF_{n,k}$. Define
\[
\comaj((g,\lambda)) = k r |Z| + \maj(g) + r |\lambda|.
\]
\end{definition}

\subsection{Generalized coinvariant algebras for $G(r,1,n)$}
Let us now introduce the coinvariant algebras studied in this paper. Recall that $\GL_n(\CC)$ acts on $\CC[\xx_n] := \CC[x_1,\ldots,x_n]$ by linear substitutions. Therefore, we can consider subring of \emph{invariant polynomials} defined by
\[
\CC[\xx_n]^{G_n} = \{f \in \CC[\xx_n] \ : \ g \cdot f(\xx_n) = f(\xx_n) \textup{ for all } g \in G_n\}.
\]
It is a classical result that this ring has algebraically independent homogeneous generators $e_1(\xx_n^r), \ldots, e_n(\xx_n^r)$, where $e_d(\xx_n^r) = e_d(x_1^r,\ldots,x_n^r)$. Therefore, the coinvariant algebra $R_n := \frac{\CC[\xx_n]}{\langle \CC[\xx_n]^{G_n}_+ \rangle}$ is equal to
\[
R_n = \frac{\CC[\xx_n]}{\langle e_1(\xx_n^r), \ldots, e_n(\xx_n^r) \rangle}.
\]
Furthermore, the Hilbert series of $R_n$ is given by
\[
\Hilb(R_n,q) = \sum_{g \in G_n} q^{\maj(g)}.
\]

In \cite{CR}, Chan and Rhoades define two generalizations of $R_n$ for all pairs of integers $0 \leq k \leq n$ with $n \geq 1$.
\begin{definition}
Let $n,k$ be integers with $n$ positive and $0 \leq k \leq n$. Let $I_{n,k}, J_{n,k} \subseteq \CC[\xx_n]$ be the ideals
\begin{align}
I_{n,k} &= \langle x_1^{kr+1}, \ldots, x_n^{kr+1}, e_{n-k+1}(\xx_n^r), \ldots, e_n(\xx_n^r) \rangle, \\
J_{n,k} &= \langle x_1^{kr}, \ldots, x_n^{kr}, e_{n-k+1}(\xx_n^r), \ldots, e_n(\xx_n^r) \rangle,
\end{align}
and let $R_{n,k}$ and $S_{n,k}$ be the corresponding quotient rings
\begin{equation}
R_{n,k} = \CC[\xx_n]/I_{n,k} \qquad \textup{and} \qquad S_{n,k} = \CC[\xx_n]/J_{n,k}.
\end{equation}
\end{definition}
We point out that this definition changes notation a little compared to Haglund, Rhoades, Shimozono \cite{HRS}. The definitions of $I_{n,k}$ and $R_{n,k}$ as in Haglund, Rhoades, Shimozono can be recovered by specializing $r = 1$ in the above definition of $J_{n,k}$ and $S_{n,k}$, rather than specializing in $I_{n,k}$ and $R_{n,k}$ as one would hope.

One of the main results \cite[Cor. 4.12]{CR} by Chan and Rhoades is the following:
\begin{theorem}
As ungraded $G_n$-modules we have $R_{n,k} \cong \CC[\FFF_{n,k}]$ and $S_{n,k} \cong \CC[\OP_{n,k}]$, where $\FFF_{n,k}$ is the set of $k$-dimensional faces in the Coxeter complex attached to $G_n$ and $\OP_{n,k}$ is the set of $r$-colored $k$-block ordered set partitions of $[n]$.
\end{theorem}

\subsection{Garsia--Stanton type bases for generalized coinvariant algebras}
Let us recall the Garsia--Stanton type bases for $R_{n,k}$ and $S_{n,k}$, as introduced by Chan and Rhoades \cite[Definitions. 5.7 \& 5.9]{CR}. In order to do so we need the classical Garsia--Stanton basis for $R_n$, indexed by elements $g \in G_n$. When $g = \pi_1^{c_1} \cdots \pi_n^{c_n}$, set $d_i(g) = \# \{j \geq i \ : \ j \in \Des(g)\}$ for the number of descents at or after position $i$. The \emph{descent monomial} $b_g$ is defined by
\[
b_g = \prod_{i=1}^n x_{\pi_i}^{r d_i(g) + c_i}.
\]
Now, the following set descends to a $\CC$-vector space bases for $S_{n,k}$ \cite[Def. 5.7 \& Thm. 5.8]{CR}:
\[
\DDD_{n,k} = \{b_g \cdot x_{\pi_1}^{r i_1} \cdots x_{\pi_{n-k}}^{r i_{n-k}} \ : \ g \in G_n, \des(g) < k, k - \des(g) > i_1 \geq \ldots \geq i_{n-k} \geq 0 \}.
\]
Furthermore, $R_{n,k}$ has a similar basis $\EEE_{n,k}$ \cite[Def. 5.9 \& Thm. 5.10]{CR} given by all elements of the form
\[
\prod_{j \in Z} x_j^{kr} \cdot b_{\pi_{z+1}^{c_{z+1}} \cdots \pi_n^{c_n}} \cdot x_{\pi_{z+1}}^{r i_{z+1}} \cdots x_{\pi_{n-k}}^{r i_{n-k}},
\]
where $Z \subseteq [n]$ satisfies $0 \leq |Z| = z \leq n-k$, $\pi_{z+1}^{c_{z+1}} \ldots \pi_n^{c_n}$ is a word on $[n] - Z$ with $\des(\pi_{z+1}^{c_{z+1}} \ldots \pi_n^{c_n}) < k$ and $k - \des(\pi_{z+1}^{c_{z+1}} \ldots \pi_n^{c_n}) > i_{z+1} \geq \ldots \geq i_{n-k} \geq 0$.

Since, $|\DDD_{n,k}| = |\OP_{n,k}|$ one might wonder whether there is a natural way to index those basis elements by elements of $\OP_{n,k}$. One way to do so is using our ascent starred model for $\OP_{n,k}$.
\begin{definition}
Given an element in $\OP_{n,k}$ represented by $(g,\lambda)$ we define
\[
b_{(g,\lambda)} := b_g \cdot x_{\pi_1}^{r i_1} \cdots x_{\pi_{n-k}}^{r i_{n-k}},
\]
where $i_j = \# \{m \ : \ \lambda_m \geq j \}$.
\end{definition}
For example, when $(g,\lambda) = (4^3 2^2 3^2 9^1 6^1 1^0 5^2 7^2 8^1,(3,2))$ is the example from before, we have $(d_1(g),\ldots,d_9(g)) = (2,2,2,2,1,1,0,0,0)$, hence
\[
b_g = x_4^{13} x_3^{12} x_2^{12} x_9^{11} x_6^6 x_1^5 x_5^2 x_7^2 x_8 \qquad \textup{and} \qquad b_{(g,\lambda)} = x_4^{21} x_3^{20} x_2^{16} x_9^{11} x_6^6 x_1^5 x_5^2 x_7^2 x_8
\]

Similarly, since $|\EEE_{n,k}| = |\FFF_{n,k}|$ we would like to index elements of $\EEE_{n,k}$ by elements of $\FFF_{n,k}$. Again, we will use our model for elements of $\FFF_{n,k}$. To this end, note that the definitions of $b_g$ and $b_{(g,\lambda)}$ make sense even if $g$ is just a word on the alphabet $\{i^j \ : \ 1 \leq i \leq n, 0 \leq j \leq r-1\}$. Therefore, we have the following definition.
\begin{definition}
Let $(Z,g,\lambda)$ represent an element in $\FFF_{n,k}$. Set
\[
b_{(Z,g,\lambda)} = \prod_{i \in Z} x_i^{kr} \cdot b_{(g,\lambda)}.
\]
\end{definition}

It is an easy check that $\DDD_{n,k} = \{b_{(g,\lambda)} \ : \ (g,\lambda) \in \OP_{n,k}\}$ and $\EEE_{n,k} = \{b_{(Z,g,\lambda)} \ : \ (Z,g,\lambda) \in \FFF_{n,k}\}$.

\subsection{Gr\"obner theory}
In this section we will recall some notions from Gr\"obner theory. Let $k$ be a field. A monomial order on $k[\xx_n]$ is a total order $<$ on the monomials $x_1^{a_1} \cdots x_n^{a_n}$ such that
\begin{itemize}
\item[1.] $1 \leq m$ for any monomial $m$;
\item[2.] if $m$ and $n$ are monomials with $m \leq n$ and $u$ is any other monomial, $um \leq un$.
\end{itemize}
Given a polynomial $0 \neq f \in k[\xx_n]$, the leading monomial $LM(f)$ is the monomial $m$ such that $m$ has nonzero coefficient in $f$ and such that any other monomial $n$ with this property has $n \leq m$. Given an ideal $I \subseteq k[\xx_n]$ let $LM(I)$ denote the ideal in $k[\xx_n]$ generated by all $LM(f)$ for $f \in I \backslash \{0\}$. We know that as a vector space $k[\xx_n]/I$ has a basis given by the classes of all monomials $m$ that are not contained in $LM(I)$ \cite[p. 248, Prop. 1]{CLOS}, which we will call the \emph{standard monomial basis for $k[\xx_n]/I$}. Note that $m \not \in LM(I)$ if and only if $m$ is not divisible by any monomial of the form $LM(f)$ for $f \in I \backslash \{0\}$.

In our case, we will equip $\CC[\yy_S]$ with the graded lexicographical monomial order with respect to the ordering of the variables by $y_S > y_T$ if $|S| > |T|$ or $|S| = |T|$ and $\min(S \backslash T) < \min(T \backslash S)$. For example, for $n=3$, this order is given by
\[
y_{\{1,2,3\}} > y_{\{1,2\}} > y_{\{1,3\}} > y_{\{2,3\}} > y_{\{1\}} > y_{\{2\}} > y_{\{3\}}.
\]
We remark that only this ordering on the variables is essential, because we will mainly work in homogeneous components of $\CC[\yy_S]$. Therefore, one could use any monomial order with this ordering on the variables instead.

\section{The quotients $\RRR_{n,k}$ and $\SSS_{n,k}$.}
\label{Bases}

We will now define the quotients of $\CC[\yy_S]$ that will be isomorphic to $R_{n,k}$ and $S_{n,k}$. Often we will show a result for the $S_{n,k}$ quotient and then most of the arguments will directly transfer over to the case of $R_{n,k}$. First, we will fix some notation. Recall that $\CC[\yy_S]$ is the polynomial ring in variables indexed by the nonempty subsets $S \subseteq [n]$.
\begin{definition}
\label{GarsiaStantonIJ}
Let $0 \leq k \leq n$ be integers with $n \geq 1$. In $\CC[\yy_S]$ we define the following ideals:
\begin{align*}
\III_{n,k} &= \langle y_S \cdot y_T, \theta_{n-k+1}, \ldots, \theta_n, y_{S_1} \cdots y_{S_{kr+1}} \rangle; \\
\JJJ_{n,k} &= \langle y_S \cdot y_T, \theta_{n-k+1}, \ldots, \theta_n, y_{S_1} \cdots y_{S_{kr}} \rangle,
\end{align*}
where $S$ and $T$ range over all pairs of nonempty subsets $S,T \subseteq [n]$ with $S \not \subseteq T$ and $T \not \subseteq S$,
\[
\theta_i = \sum_{S \subseteq [n], |S| = i} y_S^r
\]
and $(S_1,\ldots,S_{kr+1})$ and $(S_1,\ldots,S_{kr})$ range over all multichains $S_1 \subseteq \ldots \subseteq S_{kr+1}$ and $S_1 \subseteq \ldots \subseteq S_{kr}$ of nonempty subsets of $[n]$ of length $kr+1$ and $kr$ respectively.

Lastly, define $\RRR_{n,k} = \CC[\yy_S]/\III_{n,k}$ and $\SSS_{n,k} = \CC[\yy_S]/\JJJ_{n,k}$.
\end{definition}

Furthermore, let us define an important ring homomorphism $\CC[\yy_S] \rightarrow \CC[\xx_n]$.
\begin{definition}
Let $n$ be a positive integer. Let $\varphi : \CC[\yy_S] \rightarrow \CC[\xx_n]$ be the ring homomorphism defined by
\[
\varphi(y_S) = \prod_{i \in S} x_i.
\]
\end{definition}
Note that $\varphi$ does not vanish on $\langle y_S \cdot y_T \rangle$, so $\varphi$ does not induce a ring isomorphism between $\CC[\yy_S]/\langle y_S \cdot y_T \rangle$ and $\CC[\xx_n]$. However, it is well known that $\CC[\yy_S]/\langle y_S \cdot y_T \rangle$ has a basis given by multichain monomials and defining $\varphi$ on this basis yields a $G_n$-module isomorphism between $\CC[\yy_S]/\langle y_S \cdot y_T \rangle$ and $\CC[\xx_n]$ Here, if $y_S$ is a variable and $g \in G_n$ we define $g \cdot y_S = \alpha y_T$, where $g \cdot \prod_{i \in S} x_i = \alpha \prod_{j \in T} x_j$, and it is extended multiplicatively to $\CC[\yy_S]$.

However, even though we will show that there exist bases for $\RRR_{n,k}$ and $R_{n,k}$ (and $\SSS_{n,k}$ and $S_{n,k}$) such that the image of the $\yy$-variable basis under $\varphi$ is exactly the $\xx$-variable basis, the map $\varphi$ will not define a $G_n$-module isomorphism on these bases, not even in the case of the classical coinvariant algebra. Instead, $\varphi$ is often referred to as the \emph{transfer map}, indicating the analogy between the Stanley--Reisner quotients and the traditional $\xx$-variable quotients.

\subsection{An intermediate quotient}
Before we can prove our main results we will consider the quotient $\CC[\yy_S]$ by the ideal $\langle y_S \cdot y_T, \theta_{n-k+1},\ldots,\theta_n \rangle$.

\begin{definition}
Let $g \in G_n$ with $g = \sigma_1^{c_1} \cdots \sigma_n^{c_n}$, and let $d \in \ZZ_{\geq 0}^n$.
\begin{itemize}
\item[1.] Define $\tilde{b}_g = \prod_{i=1}^n y_{T_i}^{m_i}$, where $T_i = \{\sigma_j \ : \ 1 \leq j \leq i\}$ and
    \[
    m_i =
    \begin{cases}
    c_i - c_{i+1} + r & \qquad \textup{if } i < n \textup{ and } i \in \textup{Des}(g) ; \\
    c_i - c_{i+1} & \qquad \textup{if } i < n \textup{ and } i \not \in \textup{Des}(g) ; \\
    c_n & \qquad \textup{if } i = n.
    \end{cases}
    \]
\item[2.] Set $\tilde{b}_{(g,d)} = \tilde{b}_g \cdot \prod_{i=1}^n y_{T_i}^{r d_i}$.
\end{itemize}
\end{definition}

As an example let $n = 5$, $r = 3$ and $g = 4^0 2^2 5^2 3^2 1^1$. Then we have descents at positions $1$ and $3$, so $\tilde{b}_g = y_{\{4\}} y_{\{2,4,5\}}^3 y_{\{2,3,4,5\}} y_{\{1,2,3,4,5\}}$.

Let $\CC[\cB_n^*] := \frac{\CC[\yy_S]}{\langle y_S \cdot y_T \rangle}$ (where $S \not \subseteq T$ and $T \not \subseteq S$) be the \emph{Stanley--Reisner ring} of the Boolean algebra. It is easy to see that $\CC[\cB_n^*]$ has a $\CC$-basis given by multichain monomials, which are monomials of the form $y = y_{S_1} \cdots y_{S_t}$ with $\emptyset \neq S_1 \subseteq S_2 \subseteq \ldots \subseteq S_t \subseteq [n]$.

\begin{lemma}
Every multichain monomial $y = y_{S_1} \cdots y_{S_t}$ is equal to $\tilde{b}_{(g,d)}$ for a unique $(g,d) \in G_n \times \ZZ_{\geq 0}^n$.
\end{lemma}

Before we give the proof, let us illustrate the idea of the proof. Let $n = 6$, $r = 3$ and consider $y = y_{\{4\}}^5 y_{\{1,3,4\}}^7 y_{\{1,2,3,4,6\}} y_{\{1,2,3,4,5,6\}}^4$. If we want to write this in the form $\tilde{b}_{(g,d)}$ the underlying permutation of $g$ has to be $4 ab cd 5$ with $\{a,b\} = \{1,3\}$ and $\{c,d\} = \{2,6\}$. Now, note that if $ab = 31$, then $y_{\{4,3\}}$ will have exponent at least $1$ in $b_g$, either because $3$ and $1$ have different colors, or because $3$ and $1$ have the same color, which implies that $g$ has a descent at the second position. Similarly, we have $cd = 26$ and hence the underlying permutation is $413265$. Now, let $c_1, \ldots, c_6$ be the colors (of $4$, $1$, $3$, $2$, $6$ and $5$). By the above, we have $c_2 = c_3$ and $c_4 = c_5$. Note that $y_{\{1,2,3,4,5,6\}}$ has exponent $c_6$ in $b_g$, hence exponent equivalent to $c_6$ modulo $3$ in $b_{(g,d)}$. Therefore, since $0 \leq c_6 \leq 2$, we need $c_6 = 1$. Equivalently, $y_{\{1,2,3,4,6\}}$ has exponent equivalent to $c_5 - c_6$ modulo $3$ in $b_g$ (it is either $c_5 - c_6$ or $c_5 - c_6 + 3$) hence we have $c_5 - c_6 \equiv 1 \mod 3$ in $b_{(g,d)}$ as well. We conclude that $c_4 = c_5 = 2$. Similarly, $c_2 = c_3 = 0$ and $c_1 = 2$, hence the only option for $g$ is $4^2 1^0 3^0 2^2 6^2 5^1$. Note that in this case $b_g = y_{\{4\}}^2 y_{\{1,3,4\}} y_{\{1,2,3,4,6\}} y_{\{1,2,3,4,5,6\}}$, so we can uniquely write $y = b_{(g,d)}$ for $d = (1,0,2,0,0,1)$.

\begin{proof}
Suppose that our multichain monomial is of the form $y = y_{S_{i_1}}^{a_1} \cdots y_{S_{i_j}}^{a_j}$, where $1 \leq i_1 < \ldots, i_j \leq n$, $|S_{i_k}| = i_k$ for $1 \leq k \leq j$ and $a_1, \ldots, a_j > 0$. Let $S_{i_1} = \{g_1 < \ldots < g_{i_1}\}$, $S_{i_m} \backslash S_{i_{m-1}} = \{g_{i_{m-1}+1} < \ldots < g_{i_m}\}$ for $2 \leq m \leq j$ and $[n] \backslash S_{i_j} = \{g_{i_j+1} < \ldots < g_n\}$ (if this set is non-empty). Note that if $y = \tilde{b}_{(g,d)}$ then the one-line notation of the underlying permutation of $g$ has to be $g_1 g_2 \ldots g_n$ and all elements that are in the same set (from $\{g_1 < \ldots < g_{i_1}\}$, $\{g_{i_{m-1}+1} < \ldots < g_{i_m}\}$ and $\{g_{i_j+1} < \ldots < g_n\}$) need to have the same color. Let these colors be $c_1, \ldots, c_{i_j}$ and $c_{i_j+1}$ (the last one appearing only if necessary). Indeed, from the definition, if $h \in G_n$ and $h_i$ and $h_{i+1}$ have different colors then $y_{\{h_1,\ldots,h_i\}}$ has exponent $c_i-c_{i+1}$ or $c_i-c_{i+1}+r$ (depending on whether there is a descent or not) and in both cases this exponent is nonzero, so $\tilde{b}_{(h,d)}$ does not equal $y$. And if $h_i$ and $h_{i+1}$ have the same color, but $h_i > h_{i+1}$, then $y_{\{h_1,\ldots,h_i\}}$ would appear with exponent $r > 0$, so again this cannot happen. Therefore, the underlying permutation of $g$ is uniquely determined (if it exists). On the other hand, if such $c_1$ up to $c_{j}$ (and possibly $c_{j+1}$) exist, they are also uniquely determined, by a backwards inductive argument. Indeed, if $S_{i_j} \neq [n]$, then $Y_{[n]}$ has coefficient $0$ modulo $r$, hence we need $c_{j+1} = 0$, and else $S_{i_j} = [n]$ and $c_{j}$ has to be the exponent of $S_{i_j}$ taken modulo $n$. Now, suppose $c_{k}$ has been determined, then we will determine $c_{k-1}$. It is clear that we need $c_{k-1} - c_k \equiv a_{k-1} \mod r$, and since $c_{k-1}$ has to be taken from $\{0,\ldots,r-1\}$ this gives a unique choice. Now, for this choice of the colors, and the corresponding $g$, we show that there is a suitable $d \in \ZZ_{\geq 0}^n$. Note that by construction, $\tilde{b}_{g} = y_{S_{i_1}}^{b_1} \cdots y_{S_{i_j}}^{b_j}$, where $b_i \equiv a_i \mod r$. Furthermore, $b_i \in \{0,1,\ldots,r\}$. It is clear that we can get $d$ by taking $d_m = 0$ when $m \neq i_t$ and taking $d_{i_m} = (b_m-a_m)/r$ when $m \in \{1,\ldots,j\}$. Note that this is an integer by $b_m \equiv a_m \mod r$. Furthermore, it is nonnegative, since $a_m > 0$, $b_m \geq 0$, $a_m \equiv b_m \mod r$ and $b_m \in \{0,1,\ldots,r\}$ implies that $b_m \geq a_m$.
\end{proof}

Using this we can find a different basis for $\CC[\cB_n^*]$.
\begin{definition}
Let $g \in G_n$ and $d \in \ZZ_{\geq 0}$. Define
\[
\tilde{b}'_{(g,d)} = \theta_{n-k+1}^{d_{n-k+1}} \cdots \theta_n^{d_n} \tilde{b}_{(g,(d_1,\ldots,d_{n-k},0,\ldots,0))}.
\]
\end{definition}

\begin{lemma}
\label{free-basis}
The set $\{ \tilde{b}'_{(g,d)} \ : \ g \in G_n, d \in \ZZ_{\geq 0} \}$ is a $\CC$-basis for $\CC[\cB_n^*]$.
\end{lemma}

\begin{proof}
Order the basis $\tilde{b}_{(g,d)}$ according to the monomial order from above. Note that for each monomial $y$, the set of monomials $y'$ with $y' \leq y$ is finite, since any such monomial $y'$ must have $\deg(y') \leq \deg(y)$ and there are finitely many such monomials.

Now, if we expand $\tilde{b}'_{(g,d)}$ in terms of the basis $\{ \tilde{b}_{(g,d)} \}$ we find that
\[
\tilde{b}'_{(g,d)} = \tilde{b}_{(g,d)} + \textup{lower terms with respect to $<$}.
\]
Indeed, suppose $g$ has underlying permutation $g_1\cdots g_n$. Set $S_i = \{g_1,\ldots,g_i\}$. Note that if $g_i > g_{i+1}$, or $c_i \neq c_{i+1}$ then necessarily we have that $y_{S_i}$ occurs in $\tilde{b}_g$ with a positive exponent. Note that (since we only allow multichains), we have $\theta_a^b = \sum_{|S|=a} y_S^{rb}$. Now, terms in $\tilde{b}'_{(g,d)}$ correspond to picking one of the terms from each of the $\theta_a^b$ with positive $b$, in such a way that the result is still a multichain. Because of our monomial order, we should pick from larger $a$ first. Suppose we are picking a subset of size $i$ and suppose $i_t < i < i_{t+1}$ (set $i_{j+1} = n$) (we can exclude $i = i_t$, because of the multichain condition we must pick $S_i$). Then, we are asking for the largest $y_S$ with $|S| = i$ and $S_{i_t} \subseteq S \subseteq S_{i_{t+1}}$, which is $S = \{g_1,\ldots,g_i\}$, due to the fact that $g_{i_t+1}, \ldots, g_{i_{t+1}}$ all have the same color and are in increasing order, by the proof of the lemma above. Therefore, the largest possible monomial that could possibly appear is obtained by picking $y_{S_i}$ for every $i > n-k$ with $d_i > 0$. Now note that if we take this choice for all $i$ simultaneously we indeed get a multichain monomial, and this monomial is equal to $\tilde{b}_{(g,d)}$, as desired.

Therefore, $\tilde{b}'_{(g,d)}$ expands in a unitriangular way in terms of the basis $\{ \tilde{b}_{(g,d)} \}$ and because of the initial observation in this proof, it follows that $\{ \tilde{b}'_{(g,d)} \ : \ g \in G_n, d \in \ZZ_{\geq 0} \}$ is a basis for $\CC[\cB_n^*]$.
\end{proof}

For $(d_1,\ldots,d_{n-k}) = d \in \ZZ_{\geq 0}^{n-k}$ set $\tilde{b}_{(g,d)} = \tilde{b}_{(g,(d_1,\ldots,d_{n-k},0,\ldots,0))}$. Then the following is immediate.

\begin{corollary} \label{freeness}
$\CC[\cB_n^*]$ is a free $\CC[\theta_{n-k+1},\ldots,\theta_n]$-module with basis given by
\[
\{\tilde{b}_{(g,d)} \ : \ g \in G_n, d \in \ZZ_{\geq 0}^{n-k}\}.
\]
Furthermore, this set descends to a vector space basis for $\CC[\cB_n^*]/\langle \theta_{n-k+1},\ldots,\theta_n \rangle = $ \\ $\CC[\yy_S]/\langle y_S \cdot y_T, \theta_{n-k+1},\ldots,\theta_n \rangle$.
\end{corollary}

Additionally, this allows us to quickly determine a vector space basis for the quotient $\CC[\yy_S]/\langle y_S \cdot y_T, \theta_{n-k+1},\ldots,\theta_n,y_{S_1,\ldots,S_m} \rangle$, of which we will be interested in the cases $m = kr$ and $m=kr+1$. Again, the result is immediate, so the proof is omitted.
\begin{corollary} \label{multichainbases}
Let $m \in \mathbb{Z}_{>0}$ and consider $\CC[\yy_S]/\langle y_S \cdot y_T, \theta_{n-k+1},\ldots,\theta_n, y_{S_1} \cdots y_{S_m} \rangle$, where $(S,T)$ runs over all pairs with $S \not \subseteq T$ and $T \not \subseteq S$, and $(S_1, \ldots, S_m)$ runs over all $\emptyset \neq S_1 \subseteq \ldots \subseteq S_m \subseteq [n]$. This is a finite-dimensional $\CC$-vector space with basis given by all elements $\tilde{b}_{(g,d)}$ with $g \in G_n$, $d \in \ZZ_{\geq 0}^{n-k}$ and $\deg(\tilde{b}_{(g,d)}) < m$.
\end{corollary}

\subsection{Bases for the rings $\RRR_{n,k}$ and $\SSS_{n,k}$}
Note that Corollary \ref{multichainbases} yields bases for $\RRR_{n,k}$ and $\SSS_{n,k}$. In this section we will show that these bases can be indexed by elements of $\FFF_{n,k}$ and $\OP_{n,k}$ respectively. We will use the models introduced before.
\begin{definition}
\begin{itemize}
\item[1.] For $(g,\lambda) \in \OP_{n,k}$, let $\tilde{b}_{(g,\lambda)} = \tilde{b}_{g} \cdot y_{S_1}^r \cdots y_{S_t}^r$, where $S_i = \{g_i : 1 \leq j \leq \lambda_i\}$.
\item[2.] Let $(Z,g,\lambda) \in \FFF_{n,k}$. If (loosely extending the definition above) we have $\tilde{b}_{(g,\lambda)} = y_{S_1} \cdots y_{S_j}$, then set $\tilde{b}'_{(g,\lambda)} = y_{S_1 \cup Z} \cdots y_{S_j \cup Z}$. Now, set $\tilde{b}_{(Z,g,\lambda)} = y_Z^{kr-\deg(\tilde{b}_{(g,\lambda)})} \cdot \tilde{b}'_{(g,\lambda)}$.
\end{itemize}
\end{definition}

It is an easy check that $\varphi(\tilde{b}_{(g,\lambda)}) = b_{(g,\lambda)}$ and $\varphi(\tilde{b}_{(Z,g,\lambda)}) = b_{(Z,g,\lambda)}$. The main result is now the following.
\begin{theorem}
The sets $\{\tilde{b}_{(g,\lambda)} \ : \ (g,\lambda) \in \OP_{n,k}\}$ and $\{\tilde{b}_{(Z,g,\lambda)} \ : \ (Z,g,\lambda) \}$ are bases for $\SSS_{n,k}$ and $\RRR_{n,k}$ respectively.
\end{theorem}

\begin{proof}
Let us first show that there is a bijection between elements of the form $\tilde{b}_{(g,\lambda)}$ and $\tilde{b}_{(g,d)}$ with $\deg(\tilde{b}_{(g,d)}) < kr$. Note that for any partition $\lambda$ with parts at most $n-k$, we have (after extending the above definition to allow for any partition) $\tilde{b}_{(g,\lambda)} = \tilde{b}_{(g,d)}$, where $d = (d_1,\ldots,d_{n-k})$ with $d_i = \# \{j \ : \ \lambda_j = i\}$. Therefore, it suffices to show that $\lambda$ has at most $k-\des(g)-1$ parts if and only if $\deg(\tilde{b}_{(g,\lambda)}) < kr$. Now, note that if $\lambda$ has $m$ parts, we have
\begin{align*}
\deg(\tilde{b}_{(g,\lambda)}) &= \deg(\tilde{b}_g) + mr = \sum_{i=1}^{n-1} (c_i - c_{i+1} + r \cdot \chi(\textup{$i$ is a descent})) + c_n + mr \\
    &= c_1 + r \des(g) + mr = c_1 + (m + \des(g)) r,
\end{align*}
where $\chi$ is the indicator function given by $\chi(S) = 1$ if statement $S$ is true and $\chi(S) = 0$ otherwise. Now, since $c_1 \in \{0,1,\ldots,r-1\}$ we have $\deg(\tilde{b}_{(g,\lambda)}) < kr$ if and only if $m+\des(g)\leq k-1$, that is if and only if $\lambda$ has at most $k-\des(g)-1$ parts.

Similarly, we have to show that there is a bijection between elements of the form $\tilde{b}_{(Z,g,\lambda)}$ and $\tilde{b}_{(g,d)}$ with $\deg(\tilde{b}_{(g,d)}) \leq kr$. A similar calculation to above shows that $\deg(\tilde{b}_{(Z,g,\lambda)}) < kr$ if $Z = \emptyset$ and clearly $\deg(\tilde{b}_{(Z,g,\lambda)} = kr$ when $Z \neq \emptyset$, so it suffices to show that there is a bijection between elements of the form $\tilde{b}_{(Z,g,\lambda)}$ with $Z \neq \emptyset$ and $\tilde{b}_{(g,d)}$ with $\deg(\tilde{b}_{(g,d)}) = kr$. Note that $\deg(\tilde{b}_{(g,d)}) = c_1 + r(\des(g) + d_1 + \ldots + d_{n-k})$, so $\deg(\tilde{b}_{(g,d)}) = kr$ if and only if $c_1 = 0$ and $\des(g) + d_1 + \ldots + d_{n-k} = k$.

Now, given $\tilde{b}_{(g,d)}$ with $\deg(\tilde{b}_{(g,d)}) = kr$, we show that there is a unique $(Z,h,\lambda)$ such that $\tilde{b}_{(Z,h,\lambda)} = b_{(g,d)}$. Let $S$ be the smallest subset (in size) such that $y_S$ has positive exponent in $\tilde{b}_{(g,\lambda)}$. It is clear that if $(Z,h,\lambda)$ exists we must have $Z = S$. Now, suppose that $|S| > n-k$. Then in particular we have $d_1 = \ldots = d_{n-k} = 0$, and $g$ has no descents at positions $1, \ldots, n-k$. But then, using $c_1 = 0$, we have $\deg(\tilde{b}_{g,d}) = \deg(\tilde{b}_g) = c_1 + r \des(g) = r \des(g) < r (k-1)$, a contradiction. Therefore, let $z = |S|$, so that $1 \leq z \leq n-k$. Using $c_1 = 0$ and minimality of $S$, we see that $g = g_1^0 \cdots g_z^0 \cdots$ with $g_1 < \ldots < g_z$. Additionally, $d_1 = \ldots = d_{z-1} = 0$. Set $b = \tilde{b}_{(g,d)} / y_S^e$, where $e$ is the exponent of $y_S$, and write
\[
b = \prod_{i=1}^m y_{S \cup S_i},
\]
where $\emptyset \neq S_1 \subseteq \ldots \subseteq S_m \subseteq [n] \backslash S$. Note that $\prod_{i=1}^m y_{S_i} = \tilde{b}_{(h,d)}$ for $h = g_{z+1}^{c_{z+1}} \cdots g_n^{c_n}$ and $d = (d_{z+1},\ldots,d_{n-k})$. We now want to show that there is a unique $(h,\lambda)$ such that $(Z,h,\lambda) \in \FFF_{n,k}$ and $\tilde{b}_{(h,\lambda)} = \tilde{b}_{(h,d)}$. However, since $\deg(\tilde{b}_{(h,d)}) < kr$ the first part of the proof shows that indeed we can find such a $(h,\lambda)$.

Conversely, we show that $\tilde{b}_{(Z,h,\lambda)}$ is of the form $\tilde{b}_{(g,d)}$ for a unique $(g,d)$. Write $Z = \{g_1 < \ldots < g_z\}$ and let $h = h_1^{c_1} \cdots h_{n-z}^{c_{n-z}}$. It is clear that we must have $g = g_1^c \cdots g_z^c h_1^{c_1} \cdots h_{n-z}^{c_{n-z}}$ for a suitable $c$. Furthermore, since we need $\deg(\tilde{b}_g) \equiv 0 \mod r$, we in fact have to pick $c = 0$. Therefore, $g$ is uniquely determined, and hence $d$ (if it exists) is also uniquely determined. By construction, if $S_t = \{g_1,\ldots,g_t\}$, the exponent of $y_{S_t}$ in $\tilde{b}_g$ and in $\tilde{b}_{(Z,h,\lambda)}$ agree modulo $r$. Indeed, this is obvious for $t > z$, and for $t \leq z$ the choice of $c = 0$ guarantees this. Furthermore, for $t > n-k$ we still have that the exponents agree as integers (so not only modulo $r$). Therefore, it only suffices to show that for any $1 \leq t \leq n-k$ the exponent of $y_{S_t}$ in $\tilde{b}_{(Z,h,\lambda)}$ is at least the exponent of $y_{S_t}$ in $\tilde{b}_g$. Again, this is obvious for $z < t \leq n-k$. Additionally, it is clear for $1 \leq t < z$, since by construction $y_{S_t}$ has exponent $0$ in $\tilde{b}_g$. Now, for $t = z$, we are immediately okay if $y_{S_z} = y_Z$ occurs with exponent less $\{0,1,\ldots,r-1\}$ in $\tilde{b}_g$. Therefore, the only thing that might fail is that $y_Z$ occurs with exponent $r$ in $\tilde{b}_g$ but exponent $0$ in $\tilde{b}_{(Z,h,\lambda)}$. However, since $\deg(\tilde{b}_{(h,\lambda)}) < kr$, we know that $y_Z$ occurs with exponent at least $1$ in $\tilde{b}_{(Z,h,\lambda)}$ and therefore, with exponent at least $r$, as desired.
\end{proof}

\subsection{A Gr\"obner theory result}
In this section we will show that the above bases are actually the standard monomial bases with respect to the monomial order used.

\begin{theorem}
\label{Grobnertheorem1}
Let $0 \leq k \leq n$ be integers with $n \geq 1$. Then the set $\{\tilde{b}_{(g,\lambda)} \ : \ (g,\lambda) \in \OP_{n,k}\}$ is precisely the standard monomial basis for $\SSS_{n,k}$.
\end{theorem}

\begin{proof}
Since we know that the given set is a basis, it suffices to show that the standard monomial basis of $\SSS_{n,k}$ is contained in $\{\tilde{b}_{(g,\lambda)} \ : \ (g,\lambda) \in \OP_{n,k}\}$.

Similar to Braun and Olsen \cite{BO} we show that $y \in \{\tilde{b}_{(g,\lambda)} \ : \ (g,\lambda) \in \OP_{n,k}\}$ if and only if $y$ is not divisible by any of the monomials in the list below. The proof will then be completed by showing that each of these monomials occurs as the leading term of some element of $\JJJ_{n,k}$. The list of monomials is given by
\begin{itemize}
\item[1.] \makebox[3cm][l]{\quad$y_S \cdot y_T$} for $S \not \subseteq T$ and $T \not \subseteq S$;
\item[2.] \makebox[3cm][l]{\quad$y_{[m]}^r$} for $m \geq n-k+1$;
\item[3.] \makebox[3cm][l]{\quad$y_S^{r+1}$} for $|S| \geq n-k+1$;
\item[4.] \makebox[3cm][l]{\quad$y_S^r \cdot y_T$} for $S \subsetneq T$, $|S| \geq n-k+1$ and $\min(T \backslash S) > \max(S)$;
\item[5.] \makebox[3cm][l]{\quad$y_S \cdot y_T^r$} for $S \subsetneq T$, $|T| \geq n-k+1$ and $T = S \cup [\ell]$ for some $\ell$;
\item[6.] \makebox[3cm][l]{\quad$y_{S_1} \cdot y_{S_2}^r \cdot y_{S_3}$} for $S_1 \subsetneq S_2 \subsetneq S_3$, $|S_2| \geq n-k+1$ and $\max(S_2 \backslash S_1) < \min(S_3 \backslash S_2)$;
\item[7.] \makebox[3cm][l]{\quad$y_{S_1} \cdots y_{S_{kr}}$} where $S_1 \subseteq \ldots \subseteq S_{kr}$.
\end{itemize}
We will first show necessity of these conditions, then sufficiency and lastly will exhibit these monomials as leading terms in $\JJJ_{n,k}$.

\textbf{Necessity}: we will assume that $g$ is of the form $\pi_1^{c_1} \cdots \pi_n^{c_n}$. Note that if $y_S$ with $|S| \geq n-k+1$ occurs in some $\tilde{b}_{(g,\lambda)}$ then its contribution completely comes from $\tilde{b}_g$. Now, if there is no descent at position $|S|$, $y_S$ will have exponent $c_{|S|+1} - c_{|S|} \leq (r-1)-0 < r$. Furthermore, if there is a descent at position $|S|$, we have $c_{|S|+1} \geq c_{|S|}$, so $y_S$ will have exponent $r + c_{|S|+1} - c_{|S|} \leq r$. Therefore, if $y_S$ occurs with exponent at least $r$, it occurs with exponent exactly $r$, we have a descent at position $|S|$ and $c_{|S|+1} = c_{|S|}$.
\begin{itemize}
\item[1.] Each variable occurring in $\tilde{b}_{(g,\lambda)}$ is of the form $y_{S_i}$ for $1 \leq i \leq n$, where $S_i = \{\pi_1,\ldots,\pi_i\}$. Since $S_1 \subseteq S_2 \subseteq \ldots \subseteq S_n$, every two variables in $\tilde{b}_{(g,\lambda)}$ will automatically be indexed by subsets one of which is contained in the other.
\item[2.] Since $m \geq n-k+1$, $y_{[m]}^r$ would have to come from a descent of $g$ at position $m$ with $c_m = c_{m+1}$. In order to have a descent we need $\pi_{m+1} < \pi_m$. However, $\pi_m \in [m]$, hence $\pi_m \leq m$, whereas $\pi_{m+1} \in [n] \backslash [m]$, so $\pi_{m+1} \geq m+1$.
\item[3.] This was observed above
\item[4.] Suppose such a product $y_S^r \cdot y_T$ actually occurs. Since $|S| \geq n-k+1$, $y_S^r$ comes from a descent at position $|S|$ with $c_{|S|+1} = c_{|S|}$, so $\pi_{|S|} > \pi_{|S|+1}$. Since $\{\pi_1,\ldots,\pi_{|S|}\} = S$ and $\{\pi_1,\ldots,\pi_{|T|}\} = T$, we have $\min(T \backslash S) \leq \pi_{|S| + 1} < \pi_{|S|} \leq \max(S)$, which is an obvious contradiction.
\item[5.] Suppose that such a product occurs. Again, $y_T^r$ has to come from a descent at position $|T|$ with $c_{|T|} = c_{|T|+1}$, hence $\pi_{|T|} > \pi_{|T|+1}$. Note that $\pi_{|T|} \in T \backslash S \subseteq [\ell]$, so $\pi_{|T|} \leq \ell$. Furthermore, $\pi_{|T| + 1} \not \in T$, hence in particular $\pi_{|T| + 1} > \ell$, which is a contradiction.
\item[6.] Suppose such a triple product occurs. Since $|S_2| \geq n-k+1$, $y_{S_2}^r$ comes from a descent at position $|S_2|$ with $c_{|S_2|} = c_{|S_2|+1}$, so we must have $\pi_{|S_2|} > \pi_{|S_2|+1}$. However, $\pi_{|S_2|} \in S_2 \backslash S_1$ and $\pi_{|S_2|+1} \in S_3 \backslash S_2$, so by assumption we have $\pi_{|S_2|} \leq \max(S_2 \backslash S_1) < \min(S_3 \backslash S_2) \leq \pi_{|S_2|+1}$.
\item[7.] We note that
    \begin{align*}
    \deg(\tilde{b}_{(g,\lambda)}) &\leq \deg(\tilde{b}_g) + (k-\des(\sigma)-1)r = \sum_{i=1}^n m_i + (k-\des(\sigma)-1)r \\
        &= \sum_{i=1}^n \left(c_i-c_{i+1}+r \chi(\textup{$i$ is a descent})\right) + (k-\des(\sigma)-1)r \\
        &= c_1 + r \des(\sigma) + (k-\des(\sigma)-1)r = kr + c_1 - r \leq kr-1,
    \end{align*}
    where $c_{n+1} = 0$, and $\chi$ is the indicator function of the indicated event.
\end{itemize}

\textbf{Sufficiency}: Let $m = y_{S'_1} \cdots y_{S'_t}$ be a monomial not divisible by any of the above mentioned monomials. Then combining properties 1. and 7. we may assume $S'_1 \subseteq S'_2 \subseteq \ldots \subseteq S'_t$ and $t < kr$. However, we will rewrite this as $m = y_{S_1}^{t_1} \cdots y_{S_u}^{t_u}$, where $S_1 \subsetneq S_2 \subsetneq \ldots \subsetneq S_u$.

We will first construct the corresponding $g \in G_n$, after which the augmentation $\lambda$ will follow automatically. Firstly, the underlying permutation of $\sigma$ will be given by putting the elements of $S_1$ in ascending order, then the elements of $S_2 \backslash S_1$, $\ldots$, the elements of $S_u \backslash S_{u-1}$ in ascending order and finally the elements of $[n] \backslash S_u$ in ascending order. Now, we have to assign colors to each of the elements. We will give all elements of $S_1$ the same color, all elements of $S_2 \backslash S_1$ the same color, $\ldots$, all elements of $S_u \backslash S_{u-1}$ the same color and finally all elements of $[n] \backslash S_u$ the same color. We will assign these colors in reverse order. Firstly, assign color $0$ to everything in $[n] \backslash S_u$, then assign color $t_u$ to $S_u \backslash S_{u-1}$, then color $t_u + t_{u-1}$ to $S_{u-1} \backslash S_{u-2}$, $\ldots$ and finally color $t_u + t_{u-1} + \ldots + t_1$ to $S_1$. Here, everything should be interpreted modulo $n$. It is an easy check that $m = \tilde{b}_g \cdot y_{S_1}^{r v_1} \cdots y_{S_u}^{r v_u}$, where $v_1, v_2, \ldots, v_u \geq 0$.

Now, let us check that $g$ together with some appropriate $\lambda$ satisfies the condition that $m = b_{(g,\lambda)}$. Firstly, using a similar computation to above, $r \des(g) \leq \deg(\tilde{b}_g) \leq \deg(m) < kr$, hence $\des(g) < k$, as desired. So, to see that the augmented part corresponds to an appropriate $\lambda$ we have to check two things, namely that $v_j = 0$ if $|S_j| \geq n-k+1$ and that $v_1 + \ldots + v_u \leq (k - \des(g) - 1)$. For the latter, note that
\[
r(v_1 + \ldots + v_u) = \deg(m) - \deg(b_g) < kr - \deg(b_g) \leq kr - \des(g) r = (k - \des(g))r,
\]
so $v_1 + \ldots + v_u < k - \des(g)$, as desired. For the first part, note that if $m = |S_j| \geq n-k+1$, then $y_{S_j}$ has exponent at most $r$ by condition 3. Therefore, if $v_j > 0$, we need $v_j = 1$, and the exponent of $y_{S_j}$ in $b_g$ equals $0$. In particular, $\sigma_m$ and $\sigma_{m+1}$ have the same color and $\sigma_m < \sigma_{m+1}$. Furthermore, note that $y_{S_j}$ now has exponent exactly $r$, so in particular we have $S_j \neq [m]$ by condition 2. Now, we distinguish four cases.
\begin{itemize}
\item $j = u = 1$: In this case, $\sigma_m = \max(S_1) > m$ and $\sigma_{m+1} = \min([n] \backslash S_1) \leq m$, a contradiction.
\item $j = 1, u > 1$: In this case, $\sigma_m = \max(S_1)$ and $\sigma_{m+1} = \min(S_2 \backslash S_1)$. By condition 4. this implies $\sigma_m > \sigma_{m+1}$, a contradiction.
\item $1 < j < u$: Now, $\sigma_m = \max(S_j \backslash S_{j-1})$ and $\sigma_{m+1} = \min(S_{j+1} \backslash S_j)$, but then $\sigma_m < \sigma_{m+1}$ contradicts condition 6.
\item $1 < j = u$: Now $\min([n] \backslash S_u) = \sigma_{m+1} > \sigma_m$, so $[\sigma_m] \subseteq S_u$. Furthermore, $\max(S_u \backslash S_{u-1}) = \sigma_m]$ hence $S_u \backslash S_{u-1} \subseteq [\sigma_m]$ so by $[\sigma_m] \subseteq S_u$ this implies $S_u = S_{u-1} \cup [\sigma_m]$. However, this contradicts condition 5.
\end{itemize}
Therefore, we need to have $v_j = 0$ if $|S_j| \geq n-k+1$, completing this part of the proof.

\textbf{Leading monomials}:
\begin{itemize}
\item[1.] These monomials are among the generators of $\JJJ_{n,k}$.
\item[2.] These monomials are the leading monomials of $\theta_m \in \JJJ_{n,k}$.
\item[3.] Write $m = |S|$ and consider $y_S \theta_m \in \JJJ_{n,k}$. All monomials in this polynomial are of the form $y_S \cdot y_T^r$ where $|T| = m = |S|$. Note that all such products have $S$ and $T$ incomparable, except for when $T = S$. Therefore, modulo $\JJJ_{n,k}$ this equals $y_S^{r+1}$, showing that $y_S^{r+1}$ in fact occurs in $\JJJ_{n,k}$.
\item[4.] Write $m = |S|$ and consider $\theta_m \cdot y_T$. All monomials in this polynomial are of the form $y_R^r \cdot y_T$ where $|R| = m = |S|$. Modulo $J_{n,k}$ this is equal to $\sum_R y_R^r \cdot y_T$ where $R$ runs over all such subsets with $R \subseteq T$. By assumption, $S$ is the smallest such set with respect to the monomial order, hence $y_S^r \cdot y_T$ is the leading term of this monomial.
\item[5.] Let $m = |T|$ and note that $\JJJ_{n,k}$ contains $y_S \cdot \theta_T$ which modulo $\JJJ_{n,k}$ reduces to $\sum_R y_S \cdot y_R^r$ where $S \subseteq R$ and $|R| = m$. Since $T = S \cup [\ell]$ it is clear that $T$ is the lexicographically smallest such set, so this polynomial has leading monomial $y_S \cdot y_T^r$.
\item[6.] Let $m = |S_2|$ and consider $y_{S_1} y_{S_3} \cdot \theta_m$. Similarly, this equals $\sum_T y_{S_1} y_T^r y_{S_3}$ modulo $\JJJ_{n,k}$ where $T$ runs over all $m$-element subsets $S_1 \subseteq T \subseteq S_3$. By assumption, $S_2$ is the lexicographically smallest such set, hence $y_{S_1} y_{S_2} y_{S_3}$ can be obtained as a leading monomial.
\item[7.] These monomials are among the generators of $\JJJ_{n,k}$.
\end{itemize}

This completes the proof.
\end{proof}

Similarly, we have the following result.

\begin{theorem}
\label{Grobnertheorem2}
Let $0 \leq k \leq n$ be integers with $n \geq 1$. Then the set $\{\tilde{b}_{(Z,g,\lambda)} \ : \ (Z,g,\lambda) \in \FFF_{n,k}\}$ is precisely the standard monomial basis for $\RRR_{n,k}$.
\end{theorem}

\begin{proof}
Again it suffices to show that the standard monomial basis of $\RRR_{n,k}$ is contained in $\{\tilde{b}_{(Z,g,\lambda)} \ : \ (Z,g,\lambda) \in \FFF_{n,k}\}$. We will show that a monomial $y$ belongs to this set if and only if it is not divisible by any of the monomials in the exact same list as before, except that we need to change the 7th condition into
\begin{itemize}
\item[7'.] \makebox[3cm][l]{\quad$y_{S_1} \cdots y_{S_{kr+1}}$} where $S_1 \subseteq \ldots \subseteq S_{kr+1}$.
\end{itemize}
Again we will go through the steps necessity, sufficiency and show that they occur as leading monomials in $\JJJ_{n,k}$.

\textbf{Necessity:} Condition 1 is clearly still satisfied, and conditions 2-6 follow by the exact same argument, since the appropriate monomials $y_S$ with $|S| \geq n-k+1$ still have to come from the contribution of $g$ to $b_{(Z,g,\lambda)}$, since neither $Z$, nor $\lambda$ will affect the exponent of these. For condition 7', we note that $b_{(Z,g,\lambda)}$ might now have degree $kr$ (when $Z \neq \emptyset$), but will never have degree $kr+1$ or more.

\textbf{Sufficiency:} There are two cases to consider. Let $y$ be a monomial not divisible by any of the monomials specified in the list. We will show that $y$ is of the form $b_{(Z,g,\lambda)}$. If $\deg(y) < kr$, then we set $Z = \emptyset$ and use the same procedure as in Theorem \ref{Grobnertheorem1} to find the appropriate $(g,\lambda)$. If $\deg(y) = kr$, set $Z$ to be the smallest subset $S$ of $[n]$ (in size) such that $y_S$ has positive exponent in $y$. Let $e_S$ be the exponent of $S$ and set $y' = y/y_S^{e_S}$. Now, set $y''$ to be the same monomial as $y'$ where each $y_T$ is replaced by $y_{T \backslash S}$. Since $\deg(y'') < kr$, the same procedure as in Theorem \ref{Grobnertheorem1} can be used to find appropriate $(g,\lambda)$ to complete the triple $(Z,g,\lambda)$.

\textbf{Leading monomials:} For conditions 1-6 the reasoning is exactly the same, since none of them use the multichain generators of $\III_{n,k}$. For condition 7', it again follows immediately since these multichain monomials belong to the generators of $\JJJ_{n,k}$.
\end{proof}

\section{A filtration of $\RRR_{n,k}$ and $\SSS_{n,k}$.}
\label{Filtration}

We will now prove the main result of this paper, namely the following.
\begin{theorem}\label{maintheorem}
We have $G_n$-module isomorphisms $\RRR_{n,k} \cong R_{n,k}$ and $\SSS_{n,k} \cong S_{n,k}$.
\end{theorem}

To this end, we need some definitions.
\begin{definition}
\begin{itemize}
\item[1.] For $y = y_{S_1} \cdots y_{S_m} \in \CC[\yy_S]$ a multichain monomial with $S_1 \subseteq \ldots \subseteq S_m$, we let $\mu(y)$ be the partition $(|S_m|,\ldots,|S_1|)$.
\item[2.] For $m \in \CC[\xx_n]$ a monomial, we let $\mu(m)$ be the partition given by $\mu(y)$, where $y$ is the unique multichain monomial with $\varphi(y) = m$.
\end{itemize}
\end{definition}

Now, let $\unrhd$ be the dominance order on partitions. Now let $(\AAA,A)$ be $(\RRR_{n,k},R_{n,k})$ or $(\SSS_{n,k},S_{n,k})$. Now fix $d \geq 0$ and let $\mu \vDash d$. Set
\[
U_{\unrhd \mu} = \mathrm{span}\{m \ : \ \mu(m) \vDash d, \mu(m) \unrhd \mu \} \qquad \textup{and} \qquad \UUU_{\unrhd \mu} = \mathrm{span}\{y \ : \ \mu(y) \vDash d, \mu(y) \unrhd \mu \}.
\]
and define $U_{\triangleright \mu}$ and $\UUU_{\triangleright \mu}$ in a similar fashion. Let $V_{\unrhd \mu}$ be the image of $U_{\unrhd \mu}$ in $A$, $\VVV_{\unrhd \mu}$ be the image of $V_{\unrhd \mu}$ in $\AAA$ and similarly for the other 2. Now, $A$ and $\AAA$ decompose as $G_n$-modules as
\[
\bigoplus_{d \geq 0} \bigoplus_{\mu \vDash d} V_{\unrhd \mu}/V_{\triangleright \mu} \qquad \textup{and} \qquad \bigoplus_{d \geq 0} \bigoplus_{\mu \vDash d} \VVV_{\unrhd \mu}/\VVV_{\triangleright \mu},
\]
respectively. The proof of Theorem \ref{maintheorem} now follows from the lemma below.

\begin{lemma} \label{maintheoremlemma}
For each $\mu$, $V_{\unrhd \mu}/V_{\triangleright \mu}$ and $\VVV_{\unrhd \mu}/\VVV_{\triangleright \mu}$ have bases $\{b \ : \ \mu(b) = \mu\}$ and $\{\tilde{b} \ : \ \mu(\tilde{b}) = \mu\}$ respectively, where $b$ and $\tilde{b}$ belong to the Garsia-Stanton type bases mentioned before. Furthermore, the map $\tilde{b} \rightarrow b = \varphi(b)$ induces a $G_n$-module isomorphism $\VVV_{\unrhd \mu}/\VVV_{\triangleright \mu} \rightarrow V_{\unrhd \mu}/V_{\triangleright \mu}$.
\end{lemma}

This lemma in turn follows from two other lemmas, for which we need another definition.
\begin{definition}
Let $(\AAA,A) = (\SSS_{n,k},S_{n,k})$ (resp. $(\AAA,A) = (\RRR_{n,k},R_{n,k})$). Given a partition $\mu \vDash d$ with parts that are at most $n$ we say that $\mu$ is
\begin{itemize}
\item[1.] \emph{admissible} if $\mu$ has less than $kr$ (resp. $kr+1$) parts, $n-k+1 \leq i \leq n-1$ occurs at most $r$ times and $n$ occurs at most $r-1$ times.
\item[2.] \emph{semi-admissible} if $\lambda$ has less than $kr$ (resp. $kr+1$) parts, has at most $r-1$ parts equal to $n$, but some $n-k+1 \leq i \leq n-1$ occurs at least $r+1$ times.
\item[3.] \emph{non-admissible} if $\lambda$ has at least $kr$ (resp. $kr+1$) parts or has at least $r$ parts equal to $n$.
\end{itemize}
\end{definition}

For example, when $n = 6$, $k = 3$ and $r = 2$, the partitions $(5,5,2,2,2)$, $(6,5,5,5,1)$, $(6,5,4,4,2,2,2,1)$ and $(6,6,2)$ are admissable, semi-admissable, non-admissable and non-admissable respectively, both when $(\AAA,A) = (\SSS_{n,k},S_{n,k})$ and when $(\AAA,A) = (\RRR_{n,k},R_{n,k})$. However, the partition $(6,5,5,2,2,2)$ is non-admissable if $(\AAA,A) = (\SSS_{6,3},S_{6,3})$, but admissable for $(\RRR_{6,3},R_{6,3})$.

Note that $\mu$ is admissible if and only if there exists a basis element $\tilde{b}$ with $\mu(\tilde{b}) = \mu$. A \emph{move} is replacing $y_S^r$ by $y_S^r - \theta_{|S|}$ and cancelling out all non-multichain terms or replacing $x_{i_1}^r \cdots x_{i_j}^r$ by $x_{i_1}^r \cdots x_{i_j}^r - e_j(\xx_n^r)$ (for $i_1 < \ldots < i_j$), depending on what setting one is working in.

The two main lemmas are now as follows:
\begin{lemma} \label{movelemma1}
Let $y$ be a multichain monomial in $\CC[\yy_S]$ with $\mu(y) = \mu$. Then
\begin{itemize}
\item[1.] if $\mu$ is semi-admissible or non-admissible, $y = 0$ in $\AAA$.
\item[2.] if $\mu$ is admissible, one can perform a finite number of moves to find the expansion of $y$ in $\AAA$ in terms of the Garsia-Stanton type basis. Additionally, any multichain monomial $Y$ that ever appears in this process has $\mu(Y) = \mu$.
\end{itemize}
\end{lemma}
\begin{proof}
For part 1, note that since $\mu$ is semi-admissible or non-admissible, $y$ is divisible by $y_{[n]}^r$, $y_S^{r+1}$ for $n-k+1 \leq |S| \leq n-1$ or a multichain monomial of length $kr$ (resp. $kr+1$). Since the ideal we quotient out by to get $\AAA$ contains $y_S^{r+1} \equiv y_S \cdot \theta_{|S|}$ for $n-k+1 \leq |S| \leq n-1$, we see that all of $y_{[n]}^r$, $y_S^{r+1}$ and the multichain monomials belong to the ideal, hence $y = 0$ in $\AAA$.

For the second part, recall the monomial order on $\CC[\yy_S]$ from before. Now, consider a monomial $y$ with $\mu(y) = \mu$. We claim that if $y$ is not a Garsia-Stanton type monomial we can perform a move and rewrite $y$ as a $\CC$-linear combination of smaller monomials $y'$ with $\mu(y') = \mu(y)$. Indeed, since $\mu$ is admissible, a monomial $y$ that is not a basis monomial is this for one of four reasons (by the classification of monomials that are of this form given in Theorems \ref{Grobnertheorem1} and \ref{Grobnertheorem2}):
\begin{itemize}
\item[1.] $y$ is divisible by $y_{[t]}^r$ for some $n-k+1 \leq t \leq n-1$.
\item[2.] $y$ is divisible by $y_S^r y_T$ for $S \subsetneq T$, $|S| \geq n-k+1$ and $\min(T \backslash S) > \max(S)$.
\item[3.] $y$ is divisible by $y_S y_T^r$ for $S \subsetneq T$, $|T| \geq n-k+1$ and $T = S \cup [\ell]$ for some $\ell$.
\item[4.] $y$ is divisible by $y_{S_1} y_{S_2}^r y_{S_3}$ for $S_1 \subsetneq S_2 \subseteq S_3$, $|S_2| \geq n-k+1$ and $\max(S_2 \backslash S_1) < \min(S_3 \backslash S_2)$.
\end{itemize}
In these cases, apply the move, replacing $y_{[t]}^r$, $y_S^r$, $y_T^r$ and $y_{S_2}^r$ respectively. Any monomial that remains after crossing out non-multichain monomials is obtained by replacing this specific variable by some $y_R^r$, so it suffices to show that $y_R$ is smaller than the replaced monomial. In the first case, $|R| = t$, so $y_R < y_{[t]}$. In the second case, $R$ is a subset of $T$ of size $|S|$ and since $\min(T \backslash S) > \max(S)$, $S$ was the subset corresponding to the largest possible monomial over all $R$, and similarly in the other 2 cases. Therefore, we can rewrite each non-basis monomial in terms of smaller monomials, so at some point we will be left with only basis-monomials as desired.
\end{proof}

For the proof of the second lemma we need the following observation. Note that if $y$ is any monomial in $\CC[\yy_S]$ we can still define $\mu(y)$, even if $y$ is not a multichain monomial. Now, if $y$ is a non-multichain monomial, let $y'$ be the unique multichain monomial with $\varphi(y) = \varphi(y')$. We claim that $\mu(y') \triangleright \mu(y)$. Indeed, starting from $y$ we can repeatedly replace $y_A y_B$ (for $A$ and $B$ incomparable) by $y_{A \cup B} y_{A \cap B}$. Note that on the $\mu$-level this corresponds to replacing $(\ldots, |A|, \ldots, |B|, \ldots)$ by $(\ldots,|A \cup B|,\ldots,|A \cap B|,\ldots)$ which strictly increases the corresponding partition in dominance order (since $A$ and $B$ are incomparable). Note that this local replacement does not change the image under $\varphi$ and since $\mu$ increases every time we can only do this finitely many times. So we will end up with some multichain monomial and by uniqueness this is $y'$. Also, we have done at least one replacement, so indeed $\mu(y')$ is strictly larger than $\mu(y)$.

The $\xx_n$-variable analogue of the above lemma is the following.

\begin{lemma} \label{movelemma2}
Let $m$ be a monomial in $\CC[\xx_n]$ with $\mu(m) = \mu$. Then
\begin{itemize}
\item[1.] if $\mu$ is non-admissible, $m = 0$ in $A$.
\item[2.] if $\mu$ is semi-admissible, then in $A$ we can rewrite $m$ as a sum of monomials $m_{\alpha}$ with $\mu(m_{\alpha}) \triangleright \mu$.
\item[3.] if $\mu$ is admissible, then a finite number of moves can be used to rewrite $m$ as a $\CC$-linear combination of Garsia-Stanton monomials $m'$ with $\mu(m') = \mu$, together with monomials $m_{\alpha}$ with $\mu(m_{\alpha}) \triangleright \mu$. Moreover, if the moves in part 2 of Lemma \ref{movelemma1} are replacing $y_{S_1}^r$, $y_{S_2}^r$, $\ldots$, $y_{S_m}^r$ respectively, then the moves in this case are replacing $\prod_{i \in S_1} x_i^r$, $\prod_{i \in S_2} x_i^r$, $\ldots$, $\prod_{i \in S_m} x_i^r$ respectively.
\end{itemize}
\end{lemma}
\begin{proof}
Let $y$ be the multichain monomial associated to $m$.

For the first case, since $\lambda$ is non-admissible, $y$ is divisible by either $y_{[n]}^r$ or a multichain of length $kr$ (resp. $kr+1$). In the first case, $e_n(\xx_n^r) = x_1^r\cdots x_n^r$ divides $m$, hence $m = 0$ in $R_{n,k}$. In the second case, let $j$ be an element that is in the smallest $S$ such that $y_S$ occurs in the multichain. Then $x_j^{kr}$ (resp. $x_j^{kr+1}$) divides $m$ and consequently $m = 0$ in $A$.

In the second case, we have that $y$ is divisible by $y_S^{r+1}$ for some $S$ with $n-k+1 \leq S \leq n-1$. Suppose $S = \{i_1,\ldots,i_j\}$. Then apply the move by replacing $x_{i_1}^r \cdots x_{i_j}^r$. We can ``pull back'' the move to $\CC[\yy_S]$, where we replace $y_S^r$ by $y_S^r - \theta_{|S|}$, but we \emph{do not} get rid of non-multichains. Now, any monomial that occurs will contain $y_S y_T$ for $|S| = |T|$ but $S \neq T$, hence would have been removed in the $y$-setting, but in the $x$-setting these monomials remain. However, by the above observation, all of these monomials have strictly smaller $\mu$-partition, as desired.

In the third case, again ``pull back'' to the $y$-setting and do the exact same sequence of moves as in part 2 of the above lemma, but again we do not get rid of non-multichains. Instead, we replace them by multichains with the same image under $\varphi$ and again this will strictly increase the $\mu$-partition.
\end{proof}

As an example of this phenomenon, consider $r = 2$ and $\SSS_{5,4}$. In the $\yy$-variable setting, consider $y = y_{\{5\}}^3 y_{\{2,5\}}^2 y_{\{1,2,3,5\}}^2$. Note that this is not yet of the form $\tilde{b}_{(g,\lambda)}$, for example since the appearance of $y_{\{2,5\}} y_{\{1,2,3,5\}}^2$ violates condition 5 in the proof of Theorem \ref{Grobnertheorem1}. Therefore, we apply a step and replace $y_{\{1,2,3,5\}}^2$ by $y_{\{1,2,3,5\}}^2 - \theta_4$ and after getting rid of any monomial that is not a multichain monomial we find that
\[
y \equiv - y_{\{5\}}^3 y_{\{2,5\}}^2 y_{\{1,2,4,5\}}^2 - y_{\{5\}}^3 y_{\{2,5\}}^2 y_{\{2,3,4,5\}}^2.
\]
Here, the first monomial is $\tilde{b}_{(5^1 2^0 1^0 4^0 3^0),(1)}$ is of the desired form. However, the second monomial contains $y_{\{5\}} y_{\{2,5\}}^2 y_{\{2,3,4,5\}}$, which violates condition 6 in the proof of Theorem \ref{Grobnertheorem1}. Therefore, we perform a step on $y_{\{2,5\}}^2$ and get that
\[
y \equiv - y_{\{5\}}^3 y_{\{2,5\}}^2 y_{\{1,2,4,5\}}^2 + y_{\{5\}}^3 y_{\{3,5\}}^2 y_{\{2,3,4,5\}}^2 + y_{\{5\}}^3 y_{\{4,5\}}^2 y_{\{2,3,4,5\}}^2,
\]
and one can check that all monomials appearing here are indeed of the form $\tilde{b}_{(g,\lambda)}$. Also, note that we started with a monomial with $\mu$-partition $(4,4,2,2,1,1,1)$ and each monomial kept that form.

Now, in the $x$-variable setting we have to consider $x_5^7 x_2^4 x_1^2 x_3^2$. Replacing the monomial $(x_1x_2x_3x_5)^2$ by $(x_1x_2x_3x_5)^2 - e_4(x_1^2,x_2^3,x_3^2,x_4^2,x_5^2)$ we get
\[
x_5^7 x_2^4 x_1^2 x_3^2 \equiv - x_5^7 x_2^4 x_1^2 x_4^2 - x_5^7 x_2^4 x_3^2 x_4^2 - x_5^7 x_1^2 x_2^2 x_3^2 x_4^2 - x_5^5 x_2^4 x_1^2 x_3^2 x_4^2.
\]
Here, the first monomial is a generalized Garsia-Stanton monomial, the second monomial is one we have to perform another step on. Now, the last two monomials have $y$-monomial $y_{\{5\}}^5 y_{\{1,2,3,4,5\}}^2$ and $y_{\{5\}} y_{\{2,5\}}^2 y_{\{1,2,3,4,5\}}^2$ respectively, hence they have $\mu$-partitions $(5,5,1,1,1,1,1)$ and $(5,5,2,2,1)$. Now, it holds that $(5,5,1,1,1,1,1) \triangleright (4,4,2,2,1,1,1)$ and $(5,5,2,2,1) \triangleright (4,4,2,2,1,1,1)$. Therefore,
\[
x_5^7 x_2^4 x_1^2 x_3^2 \equiv - x_5^7 x_2^4 x_1^2 x_4^2 - x_5^7 x_2^4 x_3^2 x_4^2 + \textup{monomials with larger $\mu$-partition},
\]
and hence the first step of the algorithm carries out in the a way similar to the first step in the $\yy$-variable setting. Now, applying an analogous step to $x_5^7 x_2^4 x_3^2 x_4^2$ we find that
\[
x_5^7 x_2^4 x_1^2 x_3^2 \equiv - x_5^7 x_2^4 x_1^2 x_4^2 + x_5^7 x_3^4 x_2^2 x_4^2 + x_5^7 x_4^4 x_2^2 x_3^2 + \textup{monomials with larger $\mu$-partition},
\]
which indeed show that even though the expansion of $y$ and $x_5^7 x_2^4 x_1^2 x_3^2$ in the Garsia-Stanton bases are not identical, the monomials that appear and have the same $\mu$-partition as the original monomial \emph{do} coincide, and their coefficients agree.

Lemma \ref{maintheoremlemma} is now an easy application of the Lemmas \ref{movelemma1} and \ref{movelemma2}.
\begin{proof}[Proof of Lemma \ref{maintheoremlemma}]
If $\mu$ is non-admissible or semi-admissible the above lemmas show that $V_{\unrhd \mu}/V_{\triangleright \mu}$ and $\VVV_{\unrhd \mu}/\VVV_{\triangleright \mu}$ are both trivial $G_n$-modules.

Now, suppose $\mu$ is admissible. We need to show that sending $\tilde{b}$ to $b = \varphi(\tilde{b})$ (for $\tilde{b}$ a Garsia-Stanton monomial with $\mu(\tilde{b}) = \mu$) induces a $G_n$-module isomorphism. For $g \in G_n$ we can rewrite $g \cdot \tilde{b}$ in the Garsia-Stanton basis using the moves from part 2 of Lemma \ref{movelemma1}. Since the multichain monomial corresponding to $\pi b$ is given by $\pi \tilde{b}$ we can use part 3 of Lemma \ref{movelemma2} to rewrite $\pi b$ in the same way in this given basis (viewed as basis for $V_{\unrhd \mu}/V_{\triangleright \mu}$), since all the additional monomials that appear belong to $V_{\triangleright \mu}$ and hence are $0$ in the quotient.
\end{proof}

\section{Multi-graded Frobenius series}
\label{Frobenius}

Recall that the irreducible representations of the symmetric group $\symm_n$ are indexed by partitions $\lambda \vdash n$, and the representation corresponding to $\lambda$ is typically denoted $S^{\lambda}$. We have the Frobenius map $\Frob$ from the set of (equivalence classes) of representations of $\symm_n$ to the space of symmetric functions given by linear extension of $\Frob(S^{\lambda}) = s_{\lambda}(\xx)$, where $s_{\lambda}(\xx)$ is the Schur polynomial associated to $\lambda$ in an infinite variable set $\xx = (x_1,x_2,x_3,\ldots)$. Explicitly, if $M$ is a $\symm_n$-module with $M \cong \oplus_{\lambda \vdash n} (S^{\lambda})^{\oplus n_{\lambda}}$ we set
\[
\Frob(M) = \sum_{\lambda \vdash n} n_{\lambda} s_{\lambda}(\xx).
\]

Note that if $\mu$ is a partition with parts at most $n$ and if $d$ is a nonnegative integer, the subspaces $U_d = \mathrm{span}\{m \ : \ \mathrm{deg}(m) = d\}$ and $\UUU_{\mu} = \mathrm{span}\{m \ : \ \mu(m) = \mu\}$ are $\symm_n$-stable subspaces of $\CC[\xx_n]$ and $\CC[\yy_S]$ respectively, and hence so are their images $V_d$ and $\VVV_{\mu}$ in $S_{n,k}$ and $\SSS_{n,k}$ (for $r=1$) respectively. The graded Frobenius character and multi-graded Frobenius character of $S_{n,k}$ and $\SSS_{n,k}$ are
\begin{align*}
\grFrob(S_{n,k};q) &= \sum_{d=0}^{\infty} q^d \Frob(V_d); \\
\grFrob(\SSS_{n,k};t_1,\ldots,t_n) &= \sum_{\mu} t_1^{m_1(\mu)} \cdots t_n^{m_n(\mu)} \Frob(\VVV_{\mu}),
\end{align*}
where the sum is over all partitions with parts at most $n$, and $m_i(\mu)$ is the number of parts of $\mu$ equal to $i$. We can determine the graded Frobenius image $\grFrob(\SSS_{n,k};t_1,\ldots,t_n)$.

\begin{theorem}\label{Frobeniusseries}
Suppose that $r = 1$. Then
\[
\grFrob(\SSS_{n,k};t_1,\ldots,t_n) = \sum_{\substack{\alpha\models n \\ \ell(\alpha)\le k}} \left( \prod_{i\in D(\alpha)} t_i \right) \left( \sum_{ j_1+\cdots+j_{n-k} \le k-\ell(\alpha)} \prod_{i=1}^{n-k} t_i^{j_i}  \right) s_\alpha,
\]
where the sum runs over compositions $\alpha = (\alpha_1,\ldots,\alpha_m)$ of $n$ and $D(\alpha) = \{\alpha_1,\ldots,\alpha_1+\ldots+\alpha_{m-1}\}$. Furthermore, by setting $t_i = q^i$ we recover the graded Frobenius character of $S_{n,k}$, in accordance with \cite[Corollary 6.13]{HRS} and \cite[Corollary 6.3]{HR}.
\end{theorem}

\begin{proof}
We write $\beta\preceq\alpha$ if $D(\beta)\subseteq D(\alpha)$. It is well known that
\[
h_\alpha = \sum_{\beta\preceq \alpha} s_\beta.
\]

It follows from work of Garsia and Stanton \cite{GS} that
\begin{align*}
\grFrob(\CC[\cB_n^*];&t_1,\ldots,t_n) = \sum_{\gamma\models n} \left( \prod_{i\in D(\gamma)} (t_i+t_i^2+\cdots) \right) h_\gamma \\
= & \sum_{\alpha\models n} \sum_{\alpha\preceq\gamma} \left( \prod_{i\in D(\gamma)} (t_i+t_i^2+\cdots) \right) s_\alpha \\
= & \sum_{\alpha\models n} \left( \prod_{i\in D(\alpha)} (t_i+t_i^2+\cdots) \right) \sum_{\alpha\preceq\gamma} \left( \prod_{i\in D(\gamma)\setminus D(\alpha)} (1+t_i+t_i^2+\cdots) \right) s_\alpha \\
= & \sum_{\alpha\models n} \left( \prod_{i\in D(\alpha)} (t_i+t_i^2+\cdots) \right) \left( \prod_{i\in D(\alpha^c)} (1+t_i+t_i^2+\cdots) \right) s_\alpha \\
= & \left( \prod_{i=1}^n (1+t_i+t_i^2+\cdots) \right) \sum_{\alpha\models n} \left( \prod_{i\in D(\alpha)} t_i \right) s_\alpha
\end{align*}

The Hilbert series of the polynomial algebra $\CC[\theta_{n-k+1},\ldots,\theta_n]$ is
\[
\Hilb(\CC[\theta_{n-k+1},\ldots,\theta_n]; t_1,\ldots,t_n) = \prod_{i=n-k+1}^n (1+t_i+t_i^2+\cdots).
\]
Since $\CC[\cB_n^*]$ is a free module over $\CC[\theta_1,\ldots,\theta_n]$ (Corollary \ref{freeness}) and the action of $\symm_n$ on $\CC[\cB_n^*]$ is linear over $\CC[\theta_1,\ldots,\theta_n]$, we can rewrite $\grFrob (\CC[\cB_n^*];t_1,\ldots,t_n)$ as
\[
\Hilb(\CC[\theta_{n-k+1},\ldots,\theta_n];t_1,\ldots,t_n) \cdot \grFrob(\CC[\cB_n^*]/\langle\theta_{n-k+1},\ldots,\theta_n\rangle;t_1,\ldots,t_n),
\]
hence we have
\[
\grFrob(\CC[\cB_n^*]/\langle\theta_{n-k+1},\ldots,\theta_n\rangle; t_1,\ldots,t_n) = \left(\prod_{i=1}^{n-k} (1+t_i+t_i^2+\cdots) \right) \sum_{\alpha\models n} \left( \prod_{i\in D(\alpha)} t_i \right) s_\alpha.
\]
Modulo all length $k$ multichains, which results in retaining everything in degree less than $k$ and removing everything in degree $k$ and above, we have
\[
\grFrob(\SSS_{n,k};t_1,\ldots,t_n) = \sum_{\substack{\alpha\models n \\ \ell(\alpha)\le k}} \left( \prod_{i\in D(\alpha)} t_i \right) \left( \sum_{ j_1+\cdots+j_{n-k} \le k-\ell(\alpha)} \prod_{i=1}^{n-k} t_i^{j_i}  \right) s_\alpha.
\]

We can interpret each $(j_1,\ldots,j_{n-k})$ as a partition fitting inside a $(n-k) \times (k-\ell(\alpha))$ box, by letting $j_i$ be the number of rows of length $j_i$ (and by letting $(j_1,\ldots,j_{n-k})$ run we obtain all such partitions). Now, the size of the corresponding partition is equal to $j_1 + 2j_2 + \ldots + (n-k)j_{n-k}$ so setting $t_i = q^i$ yields
\[
\sum_{\alpha \models n} q^{\maj(\alpha)} \sum_{\lambda \in (n-k) \times (k - \ell(\alpha))} q^{|\lambda|} s_{\alpha} = \sum_{\alpha\models n } q^{\maj(\alpha)} \binom{n-\ell(\alpha)}{k-\ell(\alpha)}_q s_\alpha
\]
using the fact that $\sum_{\lambda \in a \times b} q^{|\lambda|} = \binom{a+b}{b}_q$. Note that this is indeed the expression for $\grFrob(S_{n,k};q)$.
\end{proof}

\begin{remark}
The Frobenius character map has an analogue for $G_n$ as well \cite[Section 2.4]{CR}. The proofs in Section \ref{Filtration} show that $\RRR_{n,k}$ and $\SSS_{n,k}$ are a refined version of the graded $G_n$-modules $R_{n,k}$ and $S_{n,k}$, in the sense that
\begin{align*}
\grFrob(\RRR_{n,k};q,q^2,\ldots,q^n) &= \grFrob(R_{n,k};q); \\
\grFrob(\SSS_{n,k};q,q^2,\ldots,q^n) &= \grFrob(S_{n,k};q),
\end{align*}
of which we just explicitly handled the case $(\SSS_{n,k},S_{n,k})$ for $r = 1$. By finding the graded Frobenius image of $\CC[\cB_n^*]$ as a $G_n$-module and factoring out $\prod_{i=1}^n (1+x_i^r+x_i^{2r}+\ldots)$ one can obtain a similar result for general $r$. Because of the relative importance of the $\symm_n$ case over the case for general $G_n$ we decided to only prove the result in the setting of the symmetric group and leave the details for general $r$ to the interested reader.
\end{remark}

\section{Conclusion}
\label{conclusion}
In this paper we studied quotients $\RRR_{n,k}$ and $\SSS_{n,k}$ of $\CC[\yy_S]$ that generalize the coinvariant algebra attached to the complex reflection group $G_n$. Furthermore, we studied bases of these rings that are naturally indexed by the set $\FFF_{n,k}$ of $k$-dimensional faces in the Coxeter complex attached to $G_n$ and by the set of $r$-colored ordered set partitions with $k$ parts respectively.

When $k = n$, these basis elements $b_g$ of the coinvariant algebra are indexed by $r$-colored permutations $g$ of $[n]$. These basis elements have the property that $\deg(b_g) = r \des(g) + c_n$ (where $c_n$ is the color of the last element of the permutation) and $\widetilde{\deg}(b_g) = \maj(g)$, where $\widetilde{\deg}(y)$ is the degree of a monomial in $\CC[\yy_S]$ if each variable has degree $\widetilde{\deg}(y_S) = |S|$. For general $k$, and $\sigma$ an $r$-colored ordered set partition with $k$ parts, the associated basis element $b_{\sigma}$ has the property that $\widetilde{\deg}(b_{\sigma}) = \comaj(\sigma)$. Therefore, it algebraically makes sense to define a descent statistic on $\OP_{n,k}$ by $\des(\sigma) = \lfloor \deg(b_g)/r \rfloor$. This raises the following question.
\begin{problem}
Find a combinatorial interpretation for this descent statistic, or possibly a complementary ascent statistic defined via either $\asc(\sigma) = (k-1) - \des(\sigma)$ or $\asc(\sigma) = (n-1) - \des(\sigma)$.
\end{problem}

Another question that remains unanswered is the analogue of a question asked by Chan and Rhoades \cite[Problem 7.1]{CR}, namely how to generalize this to arbitrary complex reflection groups $W \subseteq \GL_n(\CC)$. Just as for $r > 1$ the combinatorics of our quotient $\RRR_{n,k}$ is controlled by the $k$-dimensional faces of the Coxeter complex attached to $G_n$, one might wonder whether the following exists.
\begin{problem}
For any complex reflection group $W \subseteq \GL_n(\CC)$ and any $0 \leq k \leq n$ find a quotient $\RRR_{W,k}$ of $\CC[\yy_S]$ whose combinatorics is controlled by the $k$-dimensional faces of something resembling a Coxeter complex attached to $W$.
\end{problem}
One might hope that answering this question would answer the question posed by Chan and Rhoades, by trying to push $\RRR_{W,k}$ forward using the transfer map $\varphi$.

\section{Acknowledgements}
The author is very grateful to Jia Huang for suggesting the approach in Lemma \ref{free-basis}, pointing out that this can be used to deduce Theorem \ref{Frobeniusseries} and for his help clarifying various details. Furthermore, the author would like to thank Brendon Rhoades for suggesting this problem and his feedback on preliminary versions of this paper. The author also thanks the anonymous referee for the helpful comments.

\end{document}